\documentclass[10pt,a4paper]{article}
\usepackage{amssymb,amsthm,amsmath,epsfig,psfrag,subfigure}
\usepackage{color}
\usepackage{yfonts}

\numberwithin{equation}{section}

\numberwithin{equation}{section}

\newtheorem{Theorem}{Theorem}[section]

\newtheorem{Corollary}[Theorem]{Corollary}
\newtheorem{Lemma}[Theorem]{Lemma}

\newtheorem{Remark}[Theorem]{Remark}

\newtheorem{Definition}{Definition}[section]

\begin{document}

\title{Computing the unique CANDECOMP/PARAFAC decomposition of unbalanced tensors by
homotopy method}
\date{}

\author{ Yueh-Cheng Kuo\thanks{%
Department of Applied Mathematics, National University of Kaohsiung,
Kaohsiung 811, Taiwan; \texttt{yckuo@nuk.edu.tw} Research supported in part by the Taiwan MOST Grant 102-2115-M-390-005-MY2.} \and Tsung-Lin Lee\thanks{%
Department of Applied Mathematics, National Sun Yat-sen University,
Kaohsiung 804, Taiwan; \texttt{leetsung@math.nsysu.edu.tw} Research supported in part by the Taiwan MOST Grant 104-2115-M-110-003.}}
\maketitle

\begin{abstract}
The Candecomp/Parafac (CP) decomposition of the tensor whose maximal
dimension is greater than its rank is considered. We derive the upper bound
of rank under which the generic uniqueness of CP decomposition is
guaranteed. The bound only depends on the dimension of the tensor and the
proof is constructive. Under these conditions, an algorithm applying
homotopy continuation method is developed for computing the CP decomposition
of tensors.
\end{abstract}

%\begin{keyword}
%the Schr\"{o}dinger equation \sep
%numerical simulations \sep

% ======================================================================

\section{Introduction}

The Candecomp/Parafac decomposition has many applications in psychometrics, chemometrics,
signal processing, numerical linear algebra, computer vision, numerical
analysis, data mining, neuroscience, graph analysis, and elsewhere \cite{Carroll:1970,Harshman:1970,Kolda:2009}.
An $N$th-order tensor $\mathcal{A}\in \mathbb{R}^{I_{1}\times I_{2}\times
\cdots \times I_{N}}$ is a multidimensional or $N$-way array. A first-order tensor is a vector, a second-order tensor is a
matrix, and tensors of order three or higher are called higher-order
tensors. An $N$th-order rank-1 tensor $\mathcal{A}=[a_{i_{1},i_{2},\cdots
,i_{N}}]$ is defined as the outer product of $N$ nonzero vectors $\mathbf{u}%
_{n}\in \mathbb{R}^{I_{n}}$ for $n=1,\cdots ,N$, denoted by $\mathbf{u}%
_{1}\circ \mathbf{u}_{2}\circ \ldots \circ \mathbf{u}_{N}$. That is,
\begin{equation*}
a_{i_{1},i_{2},\cdots ,i_{N}}=u_{1,i_{1}}u_{2,i_{2}}\cdots u_{N,i_{N}},
\end{equation*}%
where $u_{j,i_{j}}$ is the $i_{j}$-th component of vector $\mathbf{u}_{j}$.
In this case, we write $\mathcal{A}=\mathbf{u}_{1}\circ \mathbf{u}_{2}\circ
\ldots \circ \mathbf{u}_{N}$. The rank of a tensor $\mathcal{A}$, denoted by
rank$(\mathcal{A})$, is the minimal number of rank-$1$ tensors that generate
$\mathcal{A}$ as their sum. Suppose that rank$(\mathcal{A})=R$, then it can
be written as
\begin{equation}   \label{decom}
\mathcal{A}=\sum_{r=1}^{R}\mathbf{u}_{r}^{(1)}\circ \mathbf{u}%
_{r}^{(2)}\circ \ldots \circ \mathbf{u}_{r}^{(N)},
\end{equation}
where $ \mathbf{u}_{r}^{(n)}\in \mathbb{R}^{I_n}$ is nonzero vector for $1\leqslant n\leqslant N$
and $1\leqslant r\leqslant R$. Let
\begin{align}\label{eqU}
U^{(n)}=[\mathbf{u}_1^{(n)},\mathbf{u}_2^{(n)},\ldots, \mathbf{u}%
_R^{(n)}]\in \mathbb{R}^{I_n\times R}, \text{ for } 1\leqslant n\leqslant N,
\end{align}
be the factor matrices of $\mathcal{A}$.
If the equation \eqref{decom} holds, we denote $\mathcal{A}=[\!| U^{(1)}, U^{(2)},\ldots, U^{(N)} |\!]$.
The decomposition \eqref{decom} is called the canonical decomposition (CAMDECOMP) \cite%
{Carroll:1970,Kruskal:1977} or the parallel factor (PARAFAC) \cite%
{Harshman:1970}. Throughout this paper, the CP decomposition refers to the CAMDECOMP/PARAFAC decomposition.

The CP decomposition in \eqref{decom} is said to be unique if for any expression
\begin{equation*}
\mathcal{A}=\sum_{r=1}^{R}\tilde{\mathbf{u}}_{r}^{(1)}\circ \tilde{\mathbf{u}}_{r}^{(2)}\circ \ldots \circ \tilde{\mathbf{u}}_{r}^{(N)},
\end{equation*}
there exists a permutation $\pi$ of $\{1, \ldots, R \}$ such that for $ 1 \leq r \leq R$,
\begin{equation*}
\mathbf{u}_{r}^{(1)}\circ \mathbf{u}_{r}^{(2)}\circ \ldots \circ \mathbf{u}_{r}^{(N)} =
\tilde{\mathbf{u}}_{\pi(r)}^{(1)}\circ \tilde{\mathbf{u}}_{\pi(r)}^{(2)}\circ \ldots \circ \tilde{\mathbf{u}}_{\pi(r)}^{(N)}.
\end{equation*}
However, the decomposition of a tensor may not be unique. The conditions that
guarantees the uniqueness of CP decomposition have been widely investigated \cite{Bocci:2013,Chiantini:2012,Kruskal:1977,Lathauwer:2006}. The
most general and well-known result on uniqueness is due to Kruskal \cite%
{Kruskal:1977,Kruskal:1989}. Kruskal's result is for real third-order
tensor. A concise proof for complex tensors was given in \cite%
{Sidiropoulos:2000}. Sidiropoulos and Bro \cite{SidiropoulosBro:2000}
extended Kruskal's result to $N$th-order tensor. They showed the sufficient
condition for uniqueness for the CP decomposition is
\begin{align*}
\sum_{n=1}^{N}\mathrm{rank}_k(U^{(n)})\geqslant 2R+N-1,
\end{align*}
where $U^{(n)}$ are given in \eqref{eqU} and $\mathrm{rank}_k(U^{(n)})$ is
defined as the maximum value $k$ such that any $k$ columns of $U^{(n)}$ are
linearly independent. \cite{Ten Berge:2002} showed that the sufficient
condition is also necessary for tensors of rank $R = 2$ and $R = 3 $, but
not for $R > 3$. \cite{Liu:2001} showed that a necessary condition for
uniqueness of the CP decomposition is
\begin{align*}
\min_{n=1,\ldots,N}\mathrm{rank}\left(U^{(1)}\odot\cdots \odot
U^{(n-1)}\odot U^{(n+1)} \odot\cdots \odot U^{(N)}\right)=R.
\end{align*}

A tensor is called {\it unbalanced} if the maximal dimension is greater than its rank.
De Lathauwer \cite{Lathauwer:2006} showed that a third-order unbalanced tensor $%
\mathcal{A}\in \mathbb{R}^{I_1\times I_2\times I_3}$ of rank $R\leqslant I_3$ has a CP
decomposition that is {\it generically unique} if
\begin{align*}
R(R-1)\leqslant I_1(I_1-1)I_2(I_2-1)/2.
\end{align*}
A fourth-order unbalanced tensor $\mathcal{A}\in \mathbb{R}^{I_1\times I_2\times
I_3\times I_4}$ with rank$(\mathcal{A})=R\leqslant I_4$ has a CP decomposition that is
{\it generically unique} if
\begin{align*}
R(R-1)\leqslant
I_1I_2I_3(3I_1I_2I_3-I_1I_2-I_1I_3-I_2I_3-I_1-I_2-I_3+3)/4.
\end{align*}
Under this sufficient condition,  simultaneous diagonalization method is provided in \cite{Lathauwer:2006}.  Note that the CP decomposition of an $N$th-order tensor $\mathcal{A}$  of rank $R$  is ``{\it generically unique}" means that the Lebesgue measure of the set
\begin{align*}
\left\{(U^{(1)},\ldots, U^{(N)}):\begin{array}{l} \text{the CP decomposition of tensor }\\
\text{$\mathcal{A}=[\!| U^{(1)},\ldots, U^{(N)} |\!] $ is not unique.} \end{array} \right\}
\end{align*}
is zero in $\mathbb{R}^{I_1\times R}\times \cdots\times \mathbb{R}^{I_N\times R}$.

%the set of the factor matrices, $U^{(n)}\in\mathbb{R}^{I_n\times R}$ for $n=1,\ldots,N$, such that the CP decomposition of  $\mathcal{A}$ in \eqref{decom} is not unique is  Lebesgue measure zero in $\mathbb{R}^{I_1\times R}\times \cdots\times \mathbb{R}^{I_N\times R}$.

For a complex unbalanced tensor,  \cite{Bocci:2013,Chiantini:2012} provided a sharp upper bound of rank to guarantee  generic uniqueness for CP decomposition.  More precisely,  for a tensor $\mathcal{A}\in \mathbb{C}^{I_{1}\times I_{2}\times \cdots \times I_{N}}$
with $I_N\geqslant R^*$, where
\begin{align}  \label{Rstar}
R^*=\prod_{i=1}^{N-1}I_{i}-\sum_{i=1}^{N-1}(I_{i}-1),
\end{align}
one has
\begin{itemize}
\item[(i)] if $\mathrm{rank}(\mathcal{A})< R^*$, then the CP decomposition of $\mathcal{A}$ is generically unique;

\item[(ii)] if $\mathrm{rank}(\mathcal{A})=R^*$, then $\mathcal{A}$ has
finitely many CP decompositions generically;

\item[(iii)] if $\mathrm{rank}(\mathcal{A})> R^*$, then $\mathcal{A}$ has
infinitely many CP decompositions generically.
\end{itemize}
Several methods have been provided for computing the CP decomposition such as, alternating least squares (ALS) \cite{Kolda:2009}, nonlinear least squares (NLS) \cite{Sorber2012, Sorber2013} and unconstrained nonlinear optimization. Those methods may take many iterations to converge and are not guaranteed to converge to the solution. The final solution heavily depends on the starting guess \cite{Kolda:2009}.

In this paper, we consider the CP decomposition of a real unbalanced tensor. The rank of a real tensor may actually be different over $\mathbb{R}$ and
$\mathbb{C}$, see \cite{Kolda:2009}.   We show that when rank$(\mathcal{A})\leqslant R^*$,  computing the CP decomposition of $\mathcal{A}$ is equivalent to solving the system of polynomial equations which are determined by the full rank factorization of the matricization of $\mathcal{A}$. Moreover, for almost all such tensors, the corresponding solutions of the system of  polynomial equations are isolated.
Based on this approach, we develop a homotopy algorithm to compute the CP decomposition.  The numerical experiments show that if rank$(\mathcal{A})<R^*$, the unique CP decomposition of tensor $\mathcal{A}$ can always be found by this method, and if rank$(\mathcal{A})=R^*$, all possible CP decompositions can also be obtained.

%A few features distinguish our work from existing studies in algebraic geometry \cite{Bocci:2013,Chiantini:2012}: (i) we are interested in the CP decomposition of tensors over $\mathbb{R}$ as opposed to the decomposition over $\mathbb{C}$. Note that the rank of a real-valued tensor may actually be different over $\mathbb{R}$ and
%$\mathbb{C}$, see\cite{Kolda:2009}. (ii) Our derivation of the uniqueness condition is constructive and provides an algorithm to compute the CP decomposition.

This paper is organized as follows. The notations and preliminary results are in section 2. In section 3 we consider the case of third-order tensor. The fourth-order tensors are discussed in section 4. Along the same course, the results can be generalized to tensors of higher order. Some numerical experiments are shown in section 5.

\section{Notations and preliminaries}

In this section, we shall introduce some definitions, notation and give some
preliminary results. Throughout this paper, we use calligraphic letters to denote
a tensor, capital letters to denote matrices, and
lowercase (bold) letters to denote scalars (vectors). For a third-order
tensor $\mathcal{A}$, the horizontal, lateral, and frontal slides of $\mathcal{A}$, denoted by $\mathcal{A}_{i::}$, $\mathcal{A}_{:j:}$ and $\mathcal{A}_{::k}$, respectively. For a matrix $A\in \mathbb{R}^{n\times m}$, $A^{\top}$ denotes the transpose of $A$, and $\mathcal{N}(A)$ denotes the null space of $A$. The symbol $\otimes$ denotes the Kronecker product. For two
matrices $A=[a_{ij}]\in \mathbb{R}^{n\times m}$, $B\in \mathbb{R}^{p\times k}$
\begin{align*}
A\otimes B=\left[%
\begin{array}{cccc}
a_{11}B & a_{12}B & \cdots & a_{1m}B \\
a_{21}B & a_{22}B & \cdots & a_{2m}B \\
\vdots & \vdots & \cdots & \vdots \\
a_{n1}B & a_{n2}B & \cdots & a_{nm}B \\
\end{array}%
\right]\in \mathbb{R}^{np\times mk}.
\end{align*}
Let $A=[\mathbf{a}_{1},\mathbf{a}_{2},\cdots,\mathbf{a}_m]\in \mathbb{R}%
^{n\times m}$ and $B=[\mathbf{b}_{1},\mathbf{b}_{2},\cdots,\mathbf{b}_m]\in
\mathbb{R}^{k\times m}$. The $\mathrm{vec}$ operator of $A$ creates a vector
$\mathrm{vec}(A)=[\mathbf{a}_{1}^{\top},\mathbf{a}_{2}^{\top},\cdots,\mathbf{%
a}_{m}^{\top}]^{\top}\in \mathbb{R}^{nm}$ and the symbol $\odot$ denotes the
Khatri-Rao (or columnwise Kronecker) product \cite{Rao:1971}
\begin{align*}
A\odot B=[\mathbf{a}_1\otimes \mathbf{b}_1,\mathbf{a}_2\otimes \mathbf{b}%
_2,\cdots,\mathbf{a}_m\otimes \mathbf{b}_m]\in \mathbb{R}^{nk\times m}.
\end{align*}
Let  $\mathbf{1}_n=[1,\cdots,1]^{\top}\in \mathbb{R}^n$, $\mathbf{0}_n=[0,\cdots,0]^{\top}\in \mathbb{R}^n$.  We use $\mathbb{I}_n$  and $0_{n\times m}$ to denote the  $n\times n$ identity matrix and $n\times m$ zero matrix, respectively.
%$\mathbf{e}_j^n\in \mathbb{R}^n$ be the $j$th column vector of identity matrix $\mathbb{I}_n\in \mathbb{R}^{n\times n}$. An $n\times m$ zero matrix is denoted by $0_{n\times m}$.
%\textcolor{blue}{
%For $N$ nonzero vectors $\mathbf{u}_n\in \mathbb{R}^{I_n}$ for $1\leqslant
%n\leqslant N$, the outer product of those $N$ vectors, denoted by $\mathbf{u}%
%_1\circ\mathbf{u}_2\circ\ldots \circ\mathbf{u}_N$, is an $I_1\times
%I_2\times \cdots \times I_N$ rank-1 tensor whose $(i_1,i_2,\cdots,i_N)$
%entry is $u_{1,i_1}u_{2,i_2}\cdots u_{N,i_N}$, where $u_{j,i_j}$ is the $i_j$%
%-th component of vector $\mathbf{u}_j$.}

The matrix products have following properties:
\begin{align}  \label{eq2.1}
\begin{array}{l}
\mathrm{vec}(AXB)=(B^{\top}\otimes A )\mathrm{vec}(X), \\
\mathrm{vec}(ADC)=(C^{\top}\odot A )\mathbf{d},\\
(A\otimes\mathbf{1}_k^{\top})\odot(\mathbf{1}_n^{\top}\otimes B)=A\otimes B,
\end{array}%
\end{align}
where $A\in \mathbb{R}^{m\times n}$, $B\in \mathbb{R}^{\ell\times k}$, $C\in \mathbb{R}^{n\times k}$, $D=\mathrm{diag}(d_1,\cdots,d_n)$ is diagonal and $\mathbf{d}%
=[d_1,\cdots,d_n]^{\top}$. %\begin{Definition}\label{def1.1}
%An $N$th-order tensor $\mathcal{A}=[a_{i_1,i_2,\cdots,i_N}]$ is rank 1 if it equals the outer product of $N$ vectors $\mathbf{u}_1,\mathbf{u}_2,\ldots,\mathbf{u}_N$. That is,
%\begin{align*}
%a_{i_1,i_2,\cdots,i_N}=u_{1,i_1}u_{2,i_2}\cdots u_{N,i_N},
%\end{align*}
%where $u_{j,i_j}$ is the $i_j$-th component of vector $\mathbf{u}_j$. The outer product of $\mathbf{u}_1,\mathbf{u}_2,\ldots,\mathbf{u}_N$ is denoted by $\mathbf{u}_1\circ\mathbf{u}_2\circ\ldots \circ\mathbf{u}_N$.
%\end{Definition}
%We now introduction the  canonical decomposition of a tensor $\mathcal{A}$ that is dealt with in this paper.
%\begin{Definition}\label{def1.2}
%A canonical decomposition (CANDECOMP) of a tensor $\mathcal{A}\in \mathbb{R}^{I_1\times I_2\times \cdots \times I_N}$ is a decomposition of $\mathcal{A}$ as a linear combination of a minimal number of rank-1 terms  \cite{Kruskal:1977}:
%\begin{align}\label{CANDECOMP}
%\mathcal{A}=\sum_{r=1}^R\lambda_r\mathbf{u}_1^{(r)}\circ\mathbf{u}_2^{(r)}\circ\ldots \circ\mathbf{u}_N^{(r)},
%\end{align}
%where $\|\mathbf{u}_i^{(r)}\|=1$ for $1\leqslant i \leqslant N$ and $1\leqslant r \leqslant R$.
%The rank of $\mathcal{A}$ (denoted by rank$(\mathcal{A})$) is the minimal number of rank-$1$ tensors, i.e.,  ${\rm rank}(\mathcal{A})=R$.
%\end{Definition}
The following well-known fact (e.g., \cite{Jiang:2009}) will be used in the proof of Lemma \ref{lem2.2}.
\begin{Lemma}\label{lem2.1}
Suppose that $h(\mathbf{x})$ is a nonzero polynomial function of variables $\mathbf{x}\in \mathbb{R}^n$. Then the zero set of $h(\mathbf{x})$, $\{\mathbf{x}\in \mathbb{R}^n|h(\mathbf{x})=0\}$, has Lebesgue measure zero in $\mathbb{R}^n$.
\end{Lemma}

%\begin{Lemma}
%\label{lem2.1}\cite{Lathauwer:2006} Consider $A\in \mathbb{R}^{I_1\times R}$
%and $B\in \mathbb{R}^{I_2\times R}$. Generically we have
%\begin{align*}
%\mathrm{rank}(A\odot B)=\min (I_1I_2,R).
%\end{align*}
%\end{Lemma}
%This lemma can be extended as follows.
\begin{Lemma}\label{lem2.2}
For generic $X=[\mathbf{x}_1,\mathbf{x}_2,\cdots,\mathbf{x}_R]\in \mathbb{R}^{I\times R}$, $Y=[\mathbf{y}_1,\mathbf{y}_2,\cdots,\mathbf{y}_R]\in \mathbb{R}^{J\times
R}$, $ Z=[\mathbf{z}_1,\mathbf{z}_2,\cdots,\mathbf{z}_R]\in \mathbb{R}^{K\times R}$, we have
\begin{itemize}
\item[(i)] $\mathrm{rank}(Y\odot X)=\min (IJ,R)$ and  $\mathrm{rank}(Z\odot Y\odot X)=\min (IJK,R)$;
\item[(ii)] {\rm rank}$(\left[
\Phi(\mathbf{x}_1,\mathbf{y}_1,\mathbf{z}_1), \cdots, \Phi(\mathbf{x}_R,%
\mathbf{y}_R,\mathbf{z}_R)%
\right])=\min
(IJK-I-J-K+2,R)$, where $\Phi:\mathbb{R}^{I+J+K}\rightarrow \mathbb{R}%
^{IJK-I-J-K+2}$ is defined by
\begin{align*}
\Phi(\tilde{\mathbf{x}},\tilde{\mathbf{y}},\tilde{\mathbf{z}})=\left[%
\begin{array}{r}
\tilde{{z}}_{1}( \tilde{\mathbf{y}}_{2}\otimes \tilde{\mathbf{x}}_{2}) \\
\tilde{{x}}_{1}( \tilde{\mathbf{z}}_{2}\otimes \tilde{\mathbf{y}}_{2}) \\
\tilde{{y}}_{1}( \tilde{\mathbf{z}}_{2}\otimes \tilde{\mathbf{x}}_{2}) \\
( \tilde{\mathbf{z}}_{2}\otimes \tilde{\mathbf{y}}_{2}\otimes \tilde{\mathbf{x}}_{2}) \\
\end{array}%
\right]\in \mathbb{R}^{IJK-I-J-K+2},
\end{align*}
and  $\tilde{\mathbf{x}}=[\tilde{x}_{1},\tilde{\mathbf{x}}_{2}^{\top}]^{\top}\in \mathbb{R}^{I}$, $%
\tilde{\mathbf{y}}=[\tilde{y}_{1},\tilde{\mathbf{y}}_{2}^{\top}]^{\top}\in \mathbb{R}^{J}$, $%
\tilde{\mathbf{z}}=[\tilde{z}_{1},\tilde{\mathbf{z}}_{2}^{\top}]^{\top}\in \mathbb{R}^{K}$.
\end{itemize}
\end{Lemma}
\begin{proof}
$(i)$ We only show $\mathrm{rank}(Y\odot X)=\min (IJ,R)$, the proof of $\mathrm{rank}(Z\odot Y\odot X)=\min (IJK,R)$ is similar.  The general case can be reduced to the $IJ\geqslant R$ case. If $IJ< R$, it suffices to prove that the result holds for any $R$ columns. Let $h(X,Y)$ be the determinant of leading $R\times R$ submatrix of $Y\odot X$. From Lemma \ref{lem2.1}, it suffices to show that $h(X,Y)$ is nonzero. Let $\widehat{X}=\mathbf{1}_J^{\top}\otimes \mathbb{I}_I\in \mathbb{R}^{I\times IJ}$, $\widehat{Y}=\mathbb{I}_J\otimes\mathbf{1}_I^{\top}\in \mathbb{R}^{J\times IJ}$, it follows from \eqref{eq2.1} that $\widehat{Y}\odot \widehat{X}=\mathbb{I}_{IJ}$. Let $X=\widehat{X}(:,1:R)$ and $Y=\widehat{Y}(:,1:R)$ be the first $R$ columns of $\widehat{X}$ and $\widehat{Y}$, respectively. It is easily seen that
%\begin{align*}
%\mathbf{y}_{I(j-1)+i}=e_j^J,\ \ \mathbf{x}_{I(j-1)+i}=e_i^I, \text{ for }1\leqslant i\leqslant I, \ 1\leqslant j\leqslant J\text{ and }I(j-1)+i\leqslant R,
%\end{align*}
%where $e_j^J\in \mathbb{R}^J$ and $e_i^I\in \mathbb{R}^I$ are $j$th and $i$th column vectors of the identity matrix $I_J$ and $I_I$, respectively.
$Y\odot X=(\widehat{Y}\odot \widehat{X})(:,1:R)=[\mathbb{I}_R,0]^{\top}\in \mathbb{R}^{\ell\times R}$ and hence $h(X,Y)$ is nonzero.

$(ii)$ Similarly, we only consider the $IJK-I-J-K+2\leqslant R$ case. Let $A(X,Y,Z)=\left[
\widehat{\Phi}(\mathbf{x}_1,\mathbf{y}_1,\mathbf{z}_1), \cdots, \widehat{\Phi}(\mathbf{x}_R,%
\mathbf{y}_R,\mathbf{z}_R)%
\right]$, where
\begin{align*}
\widehat{\Phi}(\tilde{\mathbf{x}},\tilde{\mathbf{y}},\tilde{\mathbf{z}})=\left[%
\begin{array}{r}
\tilde{\mathbf{z}}\otimes\tilde{\mathbf{y}}_{2}\otimes \tilde{\mathbf{x}}_{2} \\
\tilde{x}_{1}( \tilde{\mathbf{z}}_{2}\otimes \tilde{\mathbf{y}}_{2}) \\
\tilde{y}_{1}( \tilde{\mathbf{z}}_{2}\otimes \tilde{\mathbf{x}}_{2}) \\
\end{array}%
\right]\in \mathbb{R}^{IJK-I-J-K+2},
\end{align*}
where $\tilde{\mathbf{x}}=[\tilde{x}_{1},\tilde{\mathbf{x}}_{2}^{\top}]^{\top}\in \mathbb{R}^{I}$, $%
\tilde{\mathbf{y}}=[\tilde{y}_{1},\tilde{\mathbf{y}}_{2}^{\top}]^{\top}\in \mathbb{R}^{J}$, $%
\tilde{\mathbf{z}}=[\tilde{z}_{1},\tilde{\mathbf{z}}_{2}^{\top}]^{\top}\in \mathbb{R}^{K}$.
It is easily seen that rank$(A(X,Y,Z))={\rm rank}(\left[
\Phi(\mathbf{x}_1,\mathbf{y}_1,\mathbf{z}_1), \cdots, \Phi(\mathbf{x}_R,%
\mathbf{y}_R,\mathbf{z}_R)%
\right]).$
Let $h(X,Y,Z)$ be the determinant of leading $R\times R$ submatrix of $A(X,Y,Z)$. From Lemma \ref{lem2.1}, it suffices to show that $h(X,Y,Z)$ is nonzero. Let
\begin{align*}
\widehat{Z}&=\left[\begin{array}{c|c|c}\mathbb{I}_K \otimes \mathbf{1}^{\top}_{J-1}\otimes \mathbf{1}^{\top}_{I-1}&\begin{array}{c}\mathbf{0}^{\top}\\\hline
\mathbb{I}_{K-1} \otimes \mathbf{1}^{\top}_{J-1}\end{array}&\begin{array}{c}\mathbf{0}^{\top}\\\hline
\mathbb{I}_{K-1} \otimes \mathbf{1}^{\top}_{I-1}\end{array}\end{array}\right]\in \mathbb{R}^{K\times \ell}\\
\widehat{Y}&=\left[\begin{array}{c|c|c}
\begin{array}{c}
\mathbf{0}^{\top}\\\hline \mathbf{1}^{\top}_{K} \otimes \mathbb{I}_{J-1}\otimes \mathbf{1}^{\top}_{I-1}
\end{array}&\begin{array}{c}\mathbf{0}^{\top}\\\hline \mathbf{1}^{\top}_{K-1} \otimes \mathbb{I}_{J-1}\end{array}
& \begin{array}{c}\mathbf{1}^{\top}\\\hline 0_{(J-1)\times (K-1)(I-1)}\end{array}
\end{array}
\right]\in \mathbb{R}^{J\times \ell}\\
\widehat{X}&=\left[\begin{array}{c|c|c}
\begin{array}{c}
\mathbf{0}^{\top}\\\hline  \mathbf{1}^{\top}_{K} \otimes \mathbf{1}^{\top}_{J-1}\otimes \mathbb{I}_{I-1}
\end{array}&\begin{array}{c}\mathbf{1}^{\top}\\\hline 0_{(I-1)\times (K-1)(J-1)}\end{array}
& \begin{array}{c}\mathbf{0}^{\top}\\\hline \mathbf{1}^{\top}_{K-1} \otimes \mathbb{I}_{I-1}\end{array}
\end{array}
\right]\in \mathbb{R}^{I\times \ell},
\end{align*}
where $\ell=K(J-1)(I-1)+(K-1)(J-1)+(K-1)(I-1)=IJK-I-J-K+2$. From \eqref{eq2.1}, we have $A(\widehat{X},\widehat{Y},\widehat{Z})=\mathbb{I}_{K(J-1)(I-1)}\oplus \mathbb{I}_{(K-1)(J-1)}\oplus \mathbb{I}_{(K-1)(I-1)}$. Let   $X=\widehat{X}(:,1:R)$, $Y=\widehat{Y}(:,1:R)$ and $Z=\widehat{Z}(:,1:R)$. Then $A(X,Y,Z)=[\mathbb{I}_R,0]^{\top}\in \mathbb{R}^{\ell\times R}$ and hence, $h(X,Y,Z)$ is nonzero.
\end{proof}

\begin{Definition}
\label{def1.3} The Frobenius inner product in the vector space $\mathbb{R}%
^{I_1\times I_2\times \cdots \times I_N}$ is defined by
\begin{align*}
\langle\mathcal{A},\mathcal{B}\rangle=\sum_{i_1=1}^{I_1}\sum_{i_2=1}^{I_2}\cdots
\sum_{i_N=1}^{I_N}a_{i_1,i_2,\cdots,i_N}b_{i_1,i_2,\cdots,i_N},
\end{align*}
where $\mathcal{A}=[a_{i_1,i_2,\cdots,i_N}], \mathcal{B}=[b_{i_1,i_2,%
\cdots,i_N}]\in \mathbb{R}^{I_1\times I_2\times \cdots \times I_N}$. The
Frobenius norm of a tensor matrix $\mathcal{A}$ is defined by $\|\mathcal{A}\|_F=\sqrt{\langle\mathcal{A},\mathcal{A}\rangle}$.
\end{Definition}

\section{The third-order tensor}

\label{sec3}

\subsection{Generic uniqueness condition for CP decomposition}

\label{sec3.1}

Consider a tensor $\mathcal{A}\in \mathbb{R}^{I\times J\times K}$ of which
the CP decomposition is given by
\begin{align}  \label{3tensor}
\mathcal{A}=\sum_{r=1}^R\mathbf{x}_r\circ\mathbf{y}_r\circ\mathbf{z}%
_r,
\end{align}
where $R=\mathrm{rank}(\mathcal{A})\leqslant K$%
, and for each $r\in \{1,2,\cdots R\}$, $\mathbf{x}_{r},\mathbf{y}_{r},%
\mathbf{z}_{r}$ are generic vectors. Let $X=[\mathbf{x}_1,\mathbf{x}_2,\cdots,\mathbf{x}_R]\in \mathbb{R}^{I\times R}$, $Y=[\mathbf{y}_1,\mathbf{y}_2,\cdots,\mathbf{y}_R]\in \mathbb{R}^{J\times
R}$ and $Z=[\mathbf{z}_1,\mathbf{z}_2,\cdots,\mathbf{z}_R]\in \mathbb{R}^{K\times R}$ be the factor matrices of $\mathcal{A}$, then $X$, $Y$ and $Z$ are generic.
%\begin{align*}
%\begin{array}{cc}
%X=[\mathbf{x}_1,\mathbf{x}_2,\cdots,\mathbf{x}_R]\in \mathbb{R}^{I\times R},
%& Y=[\mathbf{y}_1,\mathbf{y}_2,\cdots,\mathbf{y}_R]\in \mathbb{R}^{J\times
%R}, \\
%Z=[\mathbf{z}_1,\mathbf{z}_2,\cdots,\mathbf{z}_R]\in \mathbb{R}^{K\times R}.
%&%
%\end{array}%
%\end{align*}
It is easily seen that for each $k\in \{1,2,\cdots ,K\}$, the $k$th frontal
slide of the tensor $\mathcal{A}$ in \eqref{3tensor} is
\begin{align*}
\mathcal{A}_{::k}=X( D_{Z,k}) Y^{\top},
\end{align*}
where $D_{Z,k}=\mathrm{diag}(Z(k,:))$ and $Z(k,:)$ is the $k$th row of
matrix $Z$. Let
\begin{align}  \label{T}
T=[\mathrm{vec}(\mathcal{A}_{::1}),\mathrm{vec}(\mathcal{A}_{::2}),\cdots ,%
\mathrm{vec}(\mathcal{A}_{::K})].
\end{align}
Conventionally $T$ is called the matricization of tensor $\mathcal{A}$.
Since the matrix $D_{Z,k}$ is diagonal, it follows from \eqref{eq2.1}
that
\begin{align}  \label{vectansor}
T=(Y\odot X) Z^{\top }\in \mathbb{R}^{IJ\times K}.
\end{align}

%Since $\mathcal{A}\in \mathbb{R}^{I\times
%J\times K}$, it follows from \eqref{Rstar} that  the critical number
%\begin{align}  \label{eq3.2}
%R^*=(I-1)(J-1)+1.
%\end{align}
%
%
%In the following, we make an assumption on the rank of tensor $\mathcal{A}$.
%Assume
%\begin{align}  \label{assumption}
%R\leqslant R^*,
%\end{align}
From the definition of rank of tensor, it is easily seen that $R\leqslant IJ$.
%From \cite{Kruskal:1989,Ten Berge:2000}, we obtain that $R\leqslant IJ$.
Because vectors, $\mathbf{x}_{r},\mathbf{y}_{r},\mathbf{z}_{r}$, in %
\eqref{3tensor} are generic, from Lemma \ref{lem2.2} $(i)$, we know
that if $R\leqslant K$, then the matrices $Y\odot X\in \mathbb{R}^{IJ\times
R}$ and $Z\in \mathbb{R}^{K\times R}$ are of full column rank. In this situation, the matricization $T$ in \eqref{T} can be factorized in the form
\begin{align}  \label{EFdec}
T=EF^{\top},
\end{align}
where $E\in \mathbb{R}^{IJ\times R}$ and $F\in \mathbb{R}^{K\times R}$ are
of full column rank. From \eqref{vectansor} and \eqref{EFdec}, we have
\begin{align}  \label{YX}
Y\odot X=EW,
\end{align}
for some nonsingular $W\in \mathbb{R}^{R\times R}$. Our goal is to find an
invertible matrix $W$ such that the columns of $EW$ are Kronecker products.
Now, we consider the following inverse problem

\begin{itemize}
\item \textbf{Problem:} Given a matrix $E\in \mathbb{R}^{IJ\times R}$ with
rank$(E)=R$. % where
%$E=VW\in \mathbb{R}^{IJ\times R}$$V\in \mathbb{R}^{IJ\times R}$ is of the form in \eqref{eq1.3}, $W\in \mathbb{R\times R}$ is an invertible matrix and%$R\leqslant R^*$ and $R^*$ is given in \eqref{eq3.2}.
Find two matrices $X\in \mathbb{R}^{I\times R}$ and $Y\in \mathbb{R}^{J\times R}$ with unit columns and an invertible matrix $W\in \mathbb{R}^{R \times R}$ such
that $Y\odot X=EW$.
\end{itemize}

\begin{Remark}
\label{rem3.1} Let $\mathcal{A}\in \mathbb{R}^{I\times J\times K}$ have form
in \eqref{3tensor} with $R\leqslant K$ and let $T\in \mathbb{R}^{IJ\times K}$
be defined in \eqref{T}. Let $E\in \mathbb{R}^{IJ\times R}$ and $F\in
\mathbb{R}^{K\times R}$ be of full column rank such that \eqref{EFdec}
holds. Suppose that $(\widetilde{X},\widetilde{Y},\widetilde{W})\in \mathbb{R%
}^{I\times R}\times \mathbb{R}^{J\times R}\times \mathbb{R}^{R\times R}$ is
the solution of this inverse problem. Let $\widetilde{Z}=F\widetilde{W}^{-\top }$. %Then there exist $\widetilde{Z}\in \mathbb{R}^{K\times R}$ and $\widetilde{\Lambda }\in \mathbb{R}^{R\times R}$ such that $\widetilde{Z}\widetilde{%
%\Lambda }=F\widetilde{W}^{-\top }$, where the columns of $\widetilde{Z}$ are
%unit and $\widetilde{\Lambda }$ is diagonal.
It is easily seen that the
equation $T=(\widetilde{Y}\odot \widetilde{X})\widetilde{%
Z}^{\top }$ holds and the tensor $\mathcal{A}$ has a CP decomposition
%\begin{equation*}
$\mathcal{A}=\sum_{r=1}^{R}\tilde{\mathbf{x}}_{r}\circ
\tilde{\mathbf{y}}_{r}\circ \tilde{\mathbf{z}}_{r},$
%\end{equation*}%
where $\widetilde{X}=[\tilde{\mathbf{x}}_{1},\cdots ,\tilde{\mathbf{x}}_{R}]$%
, $\widetilde{Y}=[\tilde{\mathbf{y}}_{1},\cdots ,\tilde{\mathbf{y}}_{R}]$ and $%
\widetilde{Z}=[\tilde{\mathbf{z}}_{1},\cdots ,\tilde{\mathbf{z}}_{R}]$. Note that if the inverse problem has a unique solution up to the permutation and scalar of columns of $\widetilde{X}$ and $\widetilde{Y}$, then
the CP decomposition of tensor $\mathcal{A}$ is unique.
\end{Remark}

Since $E\in \mathbb{R}^{IJ\times R}$ is of full column
rank, the dimension of null space of $E^{\top}$, $\mathcal{N}(E^{\top})=\{%
\mathbf{u}\in \mathbb{R}^{IJ}|E^{\top}\mathbf{u}=0\}$,
%\begin{align}  \label{NE}
%\mathcal{N}(E^{\top})=\{\mathbf{u}\in \mathbb{R}^{IJ}|E^{\top}\mathbf{u}=0\},
%\end{align}
is $IJ-R$. The following lemma is useful to solve the inverse problem.

\begin{Lemma}
\label{lem3.1} Suppose that $\tilde{\mathbf{x}}\in \mathbb{R}^I$ and $\tilde{%
\mathbf{y}}\in \mathbb{R}^J$ such that $\tilde{\mathbf{y}}\otimes \tilde{%
\mathbf{x}}$ is a vector in the column space of $E\in \mathbb{R}^{IJ\times
R} $, i.e., there exists $\tilde{\mathbf{w}}\in \mathbb{R}^R$ such that $%
\tilde{\mathbf{y}}\otimes \tilde{\mathbf{x}}=E\tilde{\mathbf{w}}$. Then for
each nonzero vector $\mathbf{u}\in \mathcal{N}(E^{\top})$, let $\mathsf{U}=%
\mathrm{vec}^{-1}(\mathbf{u})\in \mathbb{R}^{I\times J}$ (i.e., \rm{vec}$(\mathsf{%
U})=\mathbf{u}$). $(\tilde{\mathbf{x}},\tilde{\mathbf{y}})$ is a solution
of the quadratic equation $\mathbf{x}^{\top}\mathsf{U}\mathbf{y}=0$.
\end{Lemma}

\begin{proof}
Suppose that  $\tilde{\mathbf{x}}\in \mathbb{R}^I$ and $\tilde{\mathbf{y}}\in \mathbb{R}^J$ such that $\tilde{\mathbf{y}}\otimes \tilde{\mathbf{x}}=E\tilde{\mathbf{w}}$ for some vector  $\tilde{\mathbf{w}}\in \mathbb{R}^R$. Since  $\mathbf{u}\in \mathcal{N}(E^{\top})$, $(\tilde{\mathbf{y}}\otimes \tilde{\mathbf{x}})^{\top}\mathbf{u}=0$. Hence, we have
\begin{align*}
\mathbf{ \tilde{x}}^{\top}\mathsf{U}\mathbf{ \tilde{y}}=(\tilde{\mathbf{y}}^{\top} \otimes \tilde{\mathbf{x}}^{\top}) {\rm vec}(\mathsf{U})=(\tilde{\mathbf{y}}\otimes \tilde{\mathbf{x}})^{\top}\mathbf{u}=0.
\end{align*}
So, $(\tilde{\mathbf{ x}},\tilde{\mathbf{ y}})$ is a solution of the quadratic equation
$\mathbf{ x}^{\top}\mathsf{U}\mathbf{ y}=0$.
\end{proof}

Let $\{\mathbf{u}_{1},\mathbf{u}_{2},\cdots ,\mathbf{u}_{IJ-R}\}$ be a
basis of $\mathcal{N}(E^{\top })$ and let $\mathsf{U}_{i}=\mathrm{vec}^{-1}(%
\mathbf{u}_{i})\in \mathbb{R}^{I\times J}$, for $i=1,2,\ldots ,IJ-R$.
Consider the system of polynomial equations
\begin{align}  \label{eq3.5}
\left\{
\begin{array}{c}
\mathbf{x}^{\top }\mathsf{U}_{1}\mathbf{y}=0, \\
\vdots \\
\mathbf{x}^{\top }\mathsf{U}_{IJ-R}\mathbf{y}=0,%
\end{array}%
\right.
\end{align}%
where $\mathbf{x}\in \mathbb{R}^{I}$, $\mathbf{y}\in \mathbb{R}^{J}$ are
unknowns.
% and $c_1,d_1\in \mathbb{R}^I$, $c_2,d_2\in \mathbb{R}^J$ are randomly generated.
Then we have the following result.

\begin{Theorem}
\label{thm3.1} Let $E\in \mathbb{R}^{IJ\times R}$, $\tilde{\mathbf{x}}\in
\mathbb{R}^I$ and $\tilde{\mathbf{y}}\in \mathbb{R}^J$. Then $\tilde{\mathbf{%
y}}\otimes \tilde{\mathbf{x}}$ belongs to the column space of $E$ if and
only if $(\tilde{\mathbf{x}},\tilde{\mathbf{y}})$ is a solution of system %
\eqref{eq3.5}.
\end{Theorem}

\begin{proof}
(Necessity.) Suppose that $\tilde{\mathbf{y}}\otimes \tilde{\mathbf{x}}$ belongs to the column space of $E$. From Lemma \ref{lem3.1}, we obtain that $(\tilde{\mathbf{x}},\tilde{\mathbf{y}})$ is a solution of system \eqref{eq3.5}.

(Sufficiency.) Suppose that  $(\tilde{\mathbf{x}},\tilde{\mathbf{y}})$ is a solution of \eqref{eq3.5}. Then we have
\begin{align*}
(\tilde{\mathbf{y}}\otimes \tilde{\mathbf{x}})^{\top}\mathbf{u}_1=(\tilde{\mathbf{y}}\otimes \tilde{\mathbf{x}})^{\top}\mathbf{u}_2=\cdots= (\tilde{\mathbf{y}}\otimes \tilde{\mathbf{x}})^{\top}\mathbf{u}_{IJ-R}=0.
\end{align*}
Since $\{\mathbf{ u}_1,\mathbf{ u}_2,\cdots, \mathbf{ u}_{IJ-R}\}$ is a basis of $\mathcal{N}(E^{\top})$, $\tilde{\mathbf{y}}\otimes \tilde{\mathbf{x}}$ belongs to the column space of $E$.
\end{proof}

\begin{Corollary}
\label{cor3.2} Suppose that the system \eqref{eq3.5} has $R$ real solutions $\left\{ (\mathbf{x}_{r},\mathbf{y}_{r})\right\} _{r=1}^{R}$ such that $Y\odot X$ is of full column rank, where $X=[\mathbf{x}_{1},\cdots ,\mathbf{x}%
_{R}]$ and $Y=[\mathbf{y}_{1},\cdots ,\mathbf{y}_{R}]$. Then the inverse
problem is solvable.
\end{Corollary}

\begin{proof}
Since $(%
\mathbf{x}_r,\mathbf{y}_r)$ for $r\in \{1,2,\ldots,R\}$ are solutions of system \eqref{eq3.5} and the matrix $%
Y\odot X$ is of full column rank, it follows from Theorem \ref{thm3.1} that there exists an invertible matrix $W$ such that $Y\odot X=EW$.
Let $D_x$ and $D_y$ be diagonal matrices such that the columns of $\widetilde{X}=XD_x$ and $\widetilde{Y}=YD_y$ are unit. Then $\widetilde{Y}\odot \widetilde{X}=E\widetilde{W}$, where $\widetilde{W}=W(D_y\odot D_x)$. Hence, the inverse problem is solvable.
\end{proof}

%\begin{Theorem}\label{thm2.1}
%Let  $X\in \mathbb{R}^{I\times R}$, $Y\in \mathbb{R}^{J\times R}$ in \eqref{eq3.1} be the solution of the inverse problem. Then for each $i\in \{1,2,\ldots,R\}$, $(x_i,y_i)$ is a solution of system \eqref{eq3.5}.
%\end{Theorem}
%\begin{proof}
%Suppose that $X\in \mathbb{R}^{I\times R}$, $Y\in \mathbb{R}^{J\times R}$ is the solution of the inverse problem, then for each $i\in \{1,2,\ldots,R\}$, $x_i\otimes y_i$ is in the column space of $E$. Hence, we have $U^{\top}(x_i\otimes y_i)=0$, where $U$ is defined in \eqref{eq3.4}. Since vec$(\mathsf{U}_j)=u_j$, for $j=1,2,\cdots,\check{R}$, it is easily seen that
%\begin{align*}
%y_i^{\top}\mathsf{U}_jx_i=u_j^{\top}(x_i\otimes y_i)=0.
%\end{align*}
%So, $(x_i,y_i)$ is a solution of system \eqref{eq3.5}.
%\end{proof}

The solutions of system \eqref{eq3.5} can be used to factorize the tensor $%
\mathcal{A}$. If $(\mathbf{x},\mathbf{y})$ is a nonzero solution of \eqref{eq3.5} then so is $(\alpha
\mathbf{x},\beta \mathbf{y})$ for each $\alpha ,\beta \in \mathbb{R}$. We consider the system
\begin{align}  \label{eq3.6}
\left\{
\begin{array}{c}
\mathbf{x}^{\top }\mathsf{U}_{1}\mathbf{y}=0, \\
\vdots \\
\mathbf{x}^{\top }\mathsf{U}_{IJ-R}\mathbf{y}=0, \\
\mathbf{c}_{x}^{\top }\mathbf{x}=1, \\
\mathbf{c}_{y}^{\top }\mathbf{y}=1,%
\end{array}%
\right.
\end{align}%
where $\mathbf{c}_{x}\in \mathbb{R}^{I}$, $\mathbf{c}_{y}\in \mathbb{R}^{J}$
are randomly generated. It follows from Theorem \ref{thm3.1} that if $%
\mathbf{y}\otimes \mathbf{x}$ is in the column space of the matrix $E$, then
$(\alpha \mathbf{x},\beta \mathbf{y})$ is a solution of \eqref{eq3.6}, where
$\alpha =1/(\mathbf{c}_{x}^{\top }\mathbf{x})$ and $\beta =1/(\mathbf{c}%
_{y}^{\top }\mathbf{y})$. %satisfies the linear system
%\begin{align}\label{eq3.7}
%\left\{
%\begin{array}{c}
%(c_1^{\top}x)\alpha+(c_2^{\top}y)\beta=1,\\
%(d_1^{\top}x)\alpha+(d_2^{\top}y)\beta=1.\\
%\end{array}
%\right.
%\end{align}

The system \eqref{eq3.6} has $IJ-R+2$ polynomial equations in $%
I+J$ unknowns $(\mathbf{x},\mathbf{y})$, where $R=\mathrm{rank}(\mathcal{A})$.
Define the critical number
\begin{align}  \label{eq3.2}
R^*=IJ-I-J+2.
\end{align}
If $R=R^{\ast }$ then \eqref{eq3.6} has $I+J$ polynomial equations in $I+J$ unknowns.
Furthermore, if $R<R^{\ast }$ (or $R>R^{\ast }$), then the system %
\eqref{eq3.6} is an overdetermined (or an underdetermined) system. Note
that the critical number $R^{\ast }$ is the same number defined in %
\eqref{Rstar} when we consider an ($I\times J\times K$) tensor.

First, we consider the case $R\leqslant R^{\ast }$. The system \eqref{eq3.6}
has more equations than unknowns if $R<R^{\ast }$. We consider the system $P(%
\mathbf{x},\mathbf{y})=\mathbf{0}$, where
\begin{align}  \label{eq3.8}
P(\mathbf{x},\mathbf{y})=\left[
\begin{array}{c}
p_{1}(\mathbf{x},\mathbf{y}) \\
\vdots \\
p_{I+J-2}(\mathbf{x},\mathbf{y}) \\
p_{I+J-1}(\mathbf{x},\mathbf{y}) \\
p_{I+J}(\mathbf{x},\mathbf{y})%
\end{array}%
\right] \equiv \left[
\begin{array}{c}
\mathbf{x}^{\top }\mathsf{U}_{1}\mathbf{y} \\
\vdots \\
\mathbf{x}^{\top }\mathsf{U}_{I+J-2}\mathbf{y} \\
\mathbf{c}_{x}^{\top }\mathbf{x}-1 \\
\mathbf{c}_{y}^{\top }\mathbf{y}-1%
\end{array}%
\right] .
\end{align}
Then the system $P(\mathbf{x},\mathbf{y})=\mathbf{0}$ has $I+J$ polynomial equations
in $I+J$ unknowns. If $R<R^{\ast }$, the system $P(\mathbf{x},\mathbf{y})=\mathbf{0}$
is the system resulting from removing $R^{\ast }-R$ equations,
\begin{align}\label{drop}
q_{i}(\mathbf{x},\mathbf{y})=\mathbf{x}^{\top }\mathsf{U}_{i}\mathbf{y}=0,%
\text{ for }i=I+J-1,\cdots ,IJ-R,
\end{align}
from system \eqref{eq3.6}. Note that the system of polynomials $P(\mathbf{x},%
\mathbf{y})$ in \eqref{eq3.8} is governed by the coefficient matrices, $%
\mathsf{U}_{1},\mathsf{U}_{2},\cdots ,\mathsf{U}_{I+J-2}$, in which $\{%
\mathrm{vec}(\mathsf{U}_{1}),\mathrm{vec}(\mathsf{U}_{1}),\cdots ,\mathrm{vec%
}(\mathsf{U}_{I+J-2})\}$ is arbitrary linearly independent set of nullspace $%
\mathcal{N}(E^{\top })$. The following lemma is straightforward and we omit
the proof.

\begin{Lemma} \label{lem3.2}
Let $P(\mathbf{x},\mathbf{y})$ be given in \eqref{eq3.8}
and $Q_x\in \mathbb{R}^{I\times I}$, $Q_y\in \mathbb{R}^{J\times J}$ be
invertible. Then $(\tilde{\mathbf{x}},\tilde{\mathbf{y}})$ is an isolated
solution of $P(\mathbf{x},\mathbf{y})=\mathbf{0}$ if and only if $(Q_x\tilde{\mathbf{x%
}},Q_y\tilde{\mathbf{y}})$ is an isolated solution of $\widetilde{P}(\mathbf{%
x},\mathbf{y})=0$, where $\widetilde{P}(\mathbf{x},\mathbf{y})$ has the form %
\eqref{eq3.8} in which $\mathsf{U}_j$, $\mathbf{c}_x$ and $\mathbf{c}_y$ are
replaced by $Q_x^{-\top}\mathsf{U}_jQ_y^{-1}$, $Q_x^{-\top}\mathbf{c}_x$ and
$Q_y^{-\top}\mathbf{c}_y$, respectively, for $j\in \{1,\cdots,I+J-2\}$.
\end{Lemma}

\begin{Remark}
\label{rem3.2} Let $E\in \mathbb{R}^{IJ\times R}$ and  $\mathrm{vec}(\mathsf{U})\in \mathcal{N}%
(E^{\top })$. Suppose that $\mathbf{y}\otimes \mathbf{x}$ is in the column
space of $E$, where $\mathbf{x}\in \mathbb{R}^{I}$ and $\mathbf{y}\in
\mathbb{R}^{J}$, it follows from Lemma \ref{lem3.1} that $\mathbf{x}^{\top }%
\mathsf{U}\mathbf{y}=0$. Let $Q_{x}\in \mathbb{R}^{I\times I}$, $Q_{y}\in
\mathbb{R}^{J\times J}$ be invertible. Then  $\hat{%
\mathbf{y}}\otimes \hat{\mathbf{x}}\equiv (Q_{y}\mathbf{y})\otimes (Q_{x}%
\mathbf{x})$ is in the column space of $\widehat{E}\equiv (Q_{y}\otimes
Q_{x})E$ and $\hat{\mathbf{x}}^{\top }\widehat{\mathsf{U}}\hat{\mathbf{y}}=0$%
, where $\mathrm{vec}(\widehat{\mathsf{U}})$ is any vector in $\mathcal{N}(%
\widehat{E}^{\top })$. Therefore, for each $\widehat{\mathsf{U}}$ with $%
\mathrm{vec}(\widehat{\mathsf{U}})\in \mathcal{N}(\widehat{E}^{\top })$,
there exists $\mathsf{U}\in \mathbb{R}^{I\times J}$ with $\mathrm{vec}(%
\mathsf{U})\in \mathcal{N}(E^{\top })$ such that $\widehat{\mathsf{U}}%
=Q_{x}^{-\top }\mathsf{U}Q_{y}^{-1}$.
\end{Remark}

%\textcolor{blue}{
%\begin{Remark} Let $E\in \mathbb{R}^{IJ\times R}$ and let $\mathsf{U}\in
%\mathbb{R}^{I\times J}$ such that $\mathrm{vec}(\mathsf{U})\in \mathcal{N}%
%(E^{\top })$. Suppose that $\mathbf{y}\otimes \mathbf{x}$ is in the column
%space of $E$, where $\mathbf{x}\in \mathbb{R}^{I}$ and $\mathbf{y}\in
%\mathbb{R}^{J}$, it follows from Lemma \ref{lem3.1}, $\mathbf{x}^{\top }%
%\mathsf{U}\mathbf{y}=0$. Let $Q_{x}\in \mathbb{R}^{I\times I}$, $Q_{y}\in
%\mathbb{R}^{J\times J}$ be invertible. Then it is easily seen that $\hat{%
%\mathbf{y}}\otimes \hat{\mathbf{x}}\equiv (Q_{y}\mathbf{y})\otimes (Q_{x}%
%\mathbf{x})$ is in the column space of $\widehat{E}\equiv (Q_{y}\otimes
%Q_{x})E$ and $\hat{\mathbf{x}}^{\top }\widehat{\mathsf{U}}\hat{\mathbf{y}}=0$%
%, where $\mathrm{vec}(\widehat{\mathsf{U}})$ is any vector in $\mathcal{N}(%
%\widehat{E}^{\top })$. In fact, for each $\widehat{\mathsf{U}}$ with $%
%\mathrm{vec}(\widehat{\mathsf{U}})\in \mathcal{N}(\widehat{E}^{\top })$,
%there exists $\mathsf{U}\in \mathbb{R}^{I\times J}$ with $\mathrm{vec}(%
%\mathsf{U})\in \mathcal{N}(E^{\top })$ such that $\widehat{\mathsf{U}}%
%=Q_{x}^{-\top }\mathsf{U}Q_{y}^{-1}$.
%\end{Remark}}

\begin{Lemma} \label{lem3.3}
Suppose that $\mathbf{x}_r\in \mathbb{R}^I$ and $\mathbf{y}_r\in \mathbb{R}^J$, for $r=1,2,\cdots,R$, such that $\{\mathbf{y}_r\otimes \mathbf{x}_r\}_{r=1}^{R}\subset \mathbb{R}^{IJ}$ is a linearly independent
set. Let $E$ be an $IJ\times R$ matrix such that $\{\mathbf{y}_r\otimes \mathbf{x}_r\}_{r=1}^{R}$ forms a basis of column space of $E$. Then
\begin{align*}
\mathcal{N}(E^{\top})=\left\{\mathrm{vec}(\mathsf{U})|\ \mathsf{U}\in
\mathbb{R}^{I\times J} \text{ such that } \mathbf{x}_{r}^{\top}\mathsf{U}%
\mathbf{y}_{r}=0\text{ for }r=1,\cdots, R\right\}.
\end{align*}
\end{Lemma}

\begin{proof}
For any vector $\mathbf{u}\in \mathcal{N}(E^{\top})$, let $\mathsf{U}={\rm vec}^{-1}(\mathbf{u})\in \mathbb{R}^{I\times J}$, then we have
$0=(\mathbf{y}_r\otimes \mathbf{x}_r)^{\top}\mathbf{u}=\mathbf{x}_r^{\top}\mathsf{U}\mathbf{y}_r$,
for each $r=1,\cdots, R$. This proves the  inclusion.

Suppose that $\mathsf{U}\in \mathbb{R}^{I\times J}$ such that $\mathbf{x}_{r}^{\top}\mathsf{U}\mathbf{y}_{r}=0$, for each $r=1,\cdots, R$. Let $\mathbf{u}={\rm vec}(\mathsf{U})$, then we have $(\mathbf{y}_r\otimes \mathbf{x}_r)^{\top}\mathbf{u}=0$, for each $r=1,\cdots, R$. Since $\{\mathbf{y}_r\otimes \mathbf{x}_r\}_{r=1}^{R}$ forms a basis of column space of $E$, we obtain $\mathbf{u}\in \mathcal{N}(E^{\top})$.
\end{proof}

Suppose that $\mathcal{A}\in \mathbb{R}^{I\times J\times K}$ has a CP
decomposition as in \eqref{3tensor} with $R\leqslant R^{\ast }$, where $%
R^{\ast }$ is defined in \eqref{eq3.2}. Let $E\in \mathbb{R}^{IJ\times R}$
and $F\in \mathbb{R}^{K\times R}$ be of full column rank such that %
\eqref{EFdec} holds, where $T\in \mathbb{R}^{IJ\times K}$ is defined in %
\eqref{T}. Then
\begin{equation*}
\mathrm{dim}(\mathcal{N}(E^{\top }))=IJ-R\geqslant IJ-R^{\ast }=I+J-2.
\end{equation*}%
Let $\{\mathbf{u}_{1},\mathbf{u}_{2},\cdots ,\mathbf{u}_{I+J-2}\}$ be an
arbitrary linearly independent set of $\mathcal{N}(E^{\top })$ and $\mathsf{U}_{j}=\mathrm{vec}^{-1}(\mathbf{u}_{j})\in \mathbb{R}^{I\times J}$
for $j=1,2,\cdots ,I+J-2$. Then we can construct a system of polynomials
$P(\mathbf{x},\mathbf{y})$ in \eqref{eq3.8}, where $\mathbf{c}%
_{x}\in \mathbb{R}^{I}$, $\mathbf{c}_{y}\in \mathbb{R}^{J}$ are randomly
generated. From \eqref{3tensor}, \eqref{YX} and Theorem \ref{thm3.1}, we
know that
\begin{align}  \label{scal_sol}
(\hat{\mathbf{x}}_{r},\hat{\mathbf{y}}_{r})=\left( \frac{1}{\mathbf{c}%
_{x}^{\top }\mathbf{x}_{r}}\mathbf{x}_{r},\frac{1}{\mathbf{c}_{y}^{\top }%
\mathbf{y}_{r}}\mathbf{y}_{r}\right) ,\text{ for }r=1,2,\cdots ,R,
\end{align}
are real solutions of $P(\mathbf{x},\mathbf{y})=\mathbf{0}$, where $\mathbf{x}_{r}$, $%
\mathbf{y}_{r}$ are given in \eqref{3tensor}. Next, we show that those $R$ real
solutions, $\{(\hat{\mathbf{x}}_{r},\hat{\mathbf{y}}_{r})\}_{r=1}^{R}$, of $P(\mathbf{x},\mathbf{y})=\mathbf{0}$ are isolated, generically.

\begin{Theorem}
\label{thm3.4}
Suppose that $\mathcal{A}\in \mathbb{R}^{I\times J\times K}$
has a CP decomposition as in \eqref{3tensor} with $R\leqslant R^*$ and $P(%
\mathbf{x},\mathbf{y})$ has the form in \eqref{eq3.8}, where $\{\mathrm{vec}(%
\mathsf{U}_j)\}_{j=1}^{I+J-2}$ is an arbitrary linearly independent set of $%
\mathcal{N}(E^{\top})$ and $\mathbf{c}_x\in \mathbb{R}^I$, $\mathbf{c}_y\in
\mathbb{R}^J$ are randomly generated. Let $\{ (\hat{\mathbf{x}}_r,\hat{\mathbf{y}}_r)\}_{r=1}^{R}$ be defined in \eqref{scal_sol}. Then for each $r\in \{1,\cdots,R\}$, $(\hat{\mathbf{x}}_r,\hat{\mathbf{y}}_r)$ is an isolated solution of $P(\mathbf{x},\mathbf{y})=\mathbf{0}$, generically.
%
%\begin{itemize}
%\item[(i).] $(\hat{\mathbf{x}}_r,\hat{\mathbf{y}}_r)$ is an isolated
%solution of $P(\mathbf{x},\mathbf{y})=\mathbf{0}$, generically;
%
%\item[(ii).] if $(\hat{\mathbf{x}},\hat{\mathbf{y}})$ is a real solution of $%
%P(\mathbf{x},\mathbf{y})=\mathbf{0}$, then $(\hat{\mathbf{x}},\hat{\mathbf{y}})$ is
%isolated, generically.
%\end{itemize}
\end{Theorem}

\begin{proof}
For each $r\in \{1,\cdots,R\}$, we know that $(\hat{\mathbf{x}}_r,\hat{\mathbf{y}}_r)$ is a solution of $P(\mathbf{x},\mathbf{y})=\mathbf{0}$.  Now we claim that $(\hat{\mathbf{x}}_r,\hat{\mathbf{y}}_r)$ is isolated. We only prove that $(\hat{\mathbf{x}}_1,\hat{\mathbf{y}}_1)$ is isolated. It suffices to show that the Jacobian of  $P(\mathbf{x},\mathbf{y})$ at  $\mathbf{ x}=\hat{\mathbf{ x}}_1$, $\mathbf{ y}=\hat{\mathbf{ y}}_1$, denoted by $DP(\hat{\mathbf{x}}_1,\hat{\mathbf{y}}_1)$,  is invertible. From Lemma \ref{lem3.2}, we may assume $\hat{\mathbf{ x}}_1=[1,0\cdots,0]^{\top}\in \mathbb{R}^I$, $\hat{\mathbf{ y}}_1=[1,0,\cdots,0]^{\top}\in \mathbb{R}^J$ and the other $R-1$ solutions are $(\hat{\mathbf{ x}}_r,\hat{\mathbf{ y}}_r)$ for $r=2,3,\ldots, R$. For each $j\in \{1,2,\cdots,I+J-2\}$, since $p_j(\hat{\mathbf{x}}_1,\hat{\mathbf{y}}_1)=0$, we have
\begin{align}\label{eq3.9}
\mathsf{U}_j=\left[\begin{array}{cc}0&\varphi_j^{\top}\\\phi_j&\widehat{\mathsf{U}}_{j}\end{array}\right]\in \mathbb{R}^{J\times I}\text{ and }
Dp_j(\hat{\mathbf{x}}_1,\hat{\mathbf{y}}_1)=[0,\phi_j^{\top}|0,\varphi_j^{\top}].
\end{align}
%\begin{align*}
%\left[\begin{array}{c}0\\\phi_j\end{array}\right]=D_xp_j(\hat{x},\hat{y})=\mathsf{U}_j^{\top}\hat{y}\in\mathbb{R}^{I},\ \left[\begin{array}{c}0\\\varphi_j\end{array}\right]=D_yp_j(\hat{x},\hat{y})=\mathsf{U}_j\hat{x}\in\mathbb{R}^{J}.
%\end{align*}
It follows from \eqref{eq3.8} and \eqref{eq3.9} that
\begin{align*}
DP(\hat{\mathbf{x}}_1,\hat{\mathbf{y}}_1)=\left[\begin{array}{cc|cc}0&\phi_1^{\top}&0&\varphi_1^{\top}\\
\vdots&\vdots&\vdots&\vdots\\
0&\phi_{I+J-2}^{\top}&0&\varphi_{I+J-2}^{\top}\\\hline
&\mathbf{ c}_x^{\top} &&0\\
&0 &&\mathbf{ c}_y^{\top}\\\end{array}\right].
\end{align*}
Since $\mathbf{ c}_x\in \mathbb{R}^I$, $\mathbf{ c}_y\in \mathbb{R}^J$ are randomly generated, the Jacobian matrix $DP(\hat{\mathbf{x}}_1,\hat{\mathbf{y}}_1)$ is invertible if and only if the matrix
\begin{align}\label{eq3.10}
\Phi=\left[\begin{array}{c|c}\phi_1^{\top}&\varphi_1^{\top}\\
\vdots&\vdots\\
\phi_{I+J-2}^{\top}&\varphi_{I+J-2}^{\top}\\\end{array}\right]
\end{align}
is invertible. Now, we  show that $\Phi$ is invertible generically, if $\{{\rm vec}(\mathsf{U}_j)\}_{j=1}^{I+J-2}$ is an arbitrary linearly independent set of   $\mathcal{N}(E^{\top})$. Since $(\hat{\mathbf{x}}_r,\hat{\mathbf{y}}_r)$ for $r=2,3,\ldots, R$ are solutions of $P(\mathbf{x},\mathbf{y})=\mathbf{0}$ in \eqref{eq3.8}, we obtain that for each $j\in \{1,2,\cdots,I+J-2\}$, $\mathsf{U}_j$ in \eqref{eq3.9} satisfies
\begin{align}\label{eq3.11}
0&=\hat{\mathbf{x}}_r^{\top}\mathsf{U}_j\hat{\mathbf{y}}_r=[\hat{x}_{r,1},\hat{\mathbf{x}}_{r,2}^{\top}]\left[\begin{array}{cc}0&\varphi_j^{\top}\\ \phi_j&\widehat{\mathsf{U}}_{j}\end{array}\right]\left[\begin{array}{c}\hat{y}_{r,1}\\\hat{\mathbf{y}}_{r,2}\end{array}\right]\nonumber\\
&=\hat{x}_{r,1}(\varphi_j^{\top}\hat{\mathbf{y}}_{r,2})+\hat{y}_{r,1}(\hat{\mathbf{x}}_{r,2}^{\top}\phi_j)+\hat{\mathbf{x}}_{r,2}^{\top}\widehat{\mathsf{U}}_j\hat{\mathbf{y}}_{r,2}\nonumber\\
&=\hat{x}_{r,1}(\varphi_j^{\top}\hat{\mathbf{y}}_{r,2})+\hat{y}_{r,1}(\hat{\mathbf{x}}_{r,2}^{\top}\phi_j)+(\hat{\mathbf{y}}_{r,2}\otimes \hat{\mathbf{x}}_{r,2})^{\top}{\rm vec}(\widehat{\mathsf{U}}_j),
\end{align}
for $r=2,3,\ldots,R$.
Since $\hat{\mathbf{ x}}_r\in \mathbb{R}^{I}$ and $\hat{\mathbf{ y}}_r\in \mathbb{R}^J$ are generic vectors, so are $\hat{\mathbf{x}}_{r,2}\in \mathbb{R}^{I-1}$ and $\hat{\mathbf{y}}_{r,2}\in \mathbb{R}^{J-1}$. Let
\begin{align*}
\Theta=\left[\begin{array}{c}(\hat{\mathbf{y}}_{2,2}\otimes \hat{\mathbf{x}}_{2,2})^{\top}\\
\vdots\\
(\hat{\mathbf{y}}_{R,2}\otimes \hat{\mathbf{x}}_{R,2})^{\top}\\\end{array}\right]~ \text{and} ~~ \mathbf{ b}_j=\left[\begin{array}{c}\hat{x}_{2,1}(\varphi_j^{\top}\hat{\mathbf{y}}_{2,2})+\hat{y}_{2,1}(\hat{\mathbf{x}}_{2,2}^{\top}\phi_j)\\ \vdots\\\hat{x}_{R,1}(\varphi_j^{\top}\hat{\mathbf{y}}_{R,2})+\hat{y}_{R,1}(\hat{\mathbf{x}}_{R,2}^{\top}\phi_j)\end{array}\right].
\end{align*}
Then $\Theta\in \mathbb{R}^{(R-1)\times (I-1)(J-1)}$ and $\mathbf{b}_j\in \mathbb{R}^{R-1}$. From Lemma \ref{lem2.2} $(i)$ and the fact that $R\leqslant R^*=(I-1)(J-1)+1$, we have ${\rm rank}(\Theta)= R-1$. It follows from \eqref{eq3.11} that those matrices $\phi_j$, $\varphi_j$ and $\widehat{\mathsf{U}}_j$ for $j=1,2,\cdots,I+J-2$ should satisfy
\begin{align}\label{eq3.12}
\Theta{\rm vec}(\widehat{\mathsf{U}}_j)=-\mathbf{ b}_j.
\end{align}
Since ${\rm rank}(\Theta)= R-1$, the linear system \eqref{eq3.12} has a solution for arbitrary $\phi_j$, $\varphi_j$. From Lemma \ref{lem3.3}, we have
\begin{align*}
\mathcal{N}(E^{\top})=\left\{{\rm vec}\left(\left[\begin{array}{cc}0&\varphi^{\top}\\\phi&\widehat{\mathsf{U}}\end{array}\right]\right)|\varphi\in \mathbb{R}^{J-1},  \phi\in \mathbb{R}^{I-1} \text{ and }\widehat{\mathsf{U}} \text{ satisfies \eqref{eq3.12}} \right\}.
\end{align*}
 Since $\left\{{\rm vec}\left(\left[\begin{array}{cc}0&\varphi_j^{\top}\\\phi_j&\widehat{\mathsf{U}}_{j}\end{array}\right]\right)\right\}_{j=1}^{I+J-2}$ is an arbitrary linearly independent set of   $\mathcal{N}(E^{\top})$, the matrix $\Phi$ in \eqref{eq3.10} is invertible generically. Hence, $(\hat{\mathbf{x}}_1,\hat{\mathbf{y}}_1)$ is isolated.
\end{proof}

Let
\begin{align} \label{realsolset}
S_\mathbb{R}=\{(\hat{\mathbf{x}},\hat{\mathbf{y}})|(\hat{\mathbf{x}},\hat{%
\mathbf{y}})\text{ is a real isolated solution of }P(\mathbf{x},\mathbf{y}%
)=\mathbf{0}\}.
\end{align}
Assume that $S_\mathbb{R}$ has $s$ real vectors.
Note that $s\geqslant R$ because $\{(\hat{\mathbf{x}}_r,\hat{\mathbf{y}}_r)\}_{r=1}^{R}\subseteq
S_\mathbb{R}$.
Let $\widetilde{X}=[\tilde{\mathbf{x}}_{1},\cdots ,\tilde{\mathbf{x}}_{R}]$
and $\widetilde{Y}=[\tilde{\mathbf{y}}_{1},\cdots ,\tilde{\mathbf{y}}_{R}]$
, where $\{(\tilde{\mathbf{x}}_r,\tilde{\mathbf{y}}_r)\}_{r=1}^{R}\subseteq S_\mathbb{R}$ are $R$ distinct vectors. If $\widetilde{Y}\odot \widetilde{X}$ is of full
column rank, then from Corollary \ref{cor3.2} and Remark \ref{rem3.1}, we can construct a CP decomposition of $\mathcal{A}$. The following theorem can be obtained directly.
%\textcolor{blue}{Generically, $\widetilde{Y}\odot \widetilde{X}$ is of full
%column rank. From Corollary \ref{cor3.2}, Remark \ref{rem3.1} and Theorem %
%\ref{thm3.4}, we have the following result.}

\begin{Theorem}
\label{thm3.5} Let $\mathcal{A}\in \mathbb{R}^{I\times J\times K}$
have a CP decomposition as in \eqref{3tensor} with $R=R^*$. Suppose that the system of polynomials $P(\mathbf{x},\mathbf{y})=\mathbf{0}$ in \eqref{eq3.8} has only isolated solutions. Then $\mathcal{A}$
has finitely many CP decompositions. In fact, it has at most
$\frac{s!}{R!(s-R)!}$ CP decompositions, where $s$ is the number of vectors in $S_\mathbb{R}$.
\end{Theorem}

\begin{Remark}
In numerical experiments, the solutions of $P(\mathbf{x},\mathbf{y})=\mathbf{0}$ in \eqref{eq3.8} are isolated. Suppose that $R<R^{\ast }$. The solutions that correspond to the CP decomposition should satisfy $P(\mathbf{x},\mathbf{y})=\mathbf{0}$ in \eqref{eq3.8} and $R^{\ast }-R$ equations $q_{i}(\mathbf{x},
\mathbf{y})=0$ in \eqref{drop}. In this case, we only obtain $R$ real isolated solution, i.e.,  $\mathcal{A}$
has a unique CP decomposition.
\end{Remark}

Now, we consider the case $R>R^{\ast }$, i.e., the system \eqref{eq3.6} is an underdetermined system. We show that the real tensor $\mathcal{A}$ has infinitely many CP decompositions generically.
\begin{Theorem}
\label{thm_under} Suppose that $\mathcal{A}\in \mathbb{R}^{I\times J\times
K} $ has a CP decomposition as in \eqref{3tensor} with $R>R^*$ and $I\geqslant 2$. Then $%
\mathcal{A}$ has infinitely many CP decompositions generically.
\end{Theorem}
\begin{proof}
From Remark \ref{rem3.1} and Corollary \ref{cor3.2}, it suffices to show that there are infinitely many $\widetilde{X}=[\tilde{\mathbf{x}}_1,\cdots, \tilde{\mathbf{
x}}_R]$ and $\widetilde{Y}=[\tilde{\mathbf{y}}_1,\cdots, \tilde{\mathbf{ y}}_R]$ such that $%
\widetilde{Y}\odot \widetilde{X}$ is of full column rank, where
$(\tilde{\mathbf{x}}_r,\tilde{\mathbf{y}}_r)$ for $r\in \{1,2,\ldots,R\}$ are real solutions of the underdetermined  system \eqref{eq3.6}. We only consider the $R=R^{\ast}+1$ case because if $R>R^{\ast}+1$ then the number of equations in  \eqref{eq3.6} is less than the  $R=R^{\ast}+1$ case. Since the tensor $\mathcal{A}$ has a CP decomposition as in \eqref{3tensor} and the factor matrices  $X=[\mathbf{x}_1,\cdots, \mathbf{
x}_R]$ and $Y=[\mathbf{y}_1,\cdots, \mathbf{ y}_R]$ are generic, it follows from Lemma \ref{lem2.2} that $%
Y\odot X$ is of full column rank. It is easily seen that for each $r\in \{1,2,\ldots,R\}$, $%
(\hat{\mathbf{x}}_r,\hat{\mathbf{y}}_r)\equiv (\mathbf{x}_r/(\mathbf{c}_x^{\top}\mathbf{x}_r) , \mathbf{y}_r/(\mathbf{
c}_y^{\top}\mathbf{ y}_r))$  is a solution of \eqref{eq3.6}. Let $\widehat{X}=[\hat{\mathbf{x}}_1,\cdots, \hat{\mathbf{
x}}_R]$ and $\hat{Y}=[\hat{\mathbf{y}}_1,\cdots, \hat{\mathbf{ y}}_R]$, we have $%
\widehat{Y}\odot \widehat{X}$ is of full column rank.

Now, we show that in the $R=R^{\ast}+1$ case,  the Jacobian  matrix $DP(\hat{\mathbf{x}}_1,\hat{\mathbf{y}}_1)\in \mathbb{R}^{(I+J-1)\times (I+J)}$ is of full row rank. From Lemma \ref{lem3.2}, we may assume $\hat{\mathbf{ x}}_1=[1,0\cdots,0]^{\top}\in \mathbb{R}^I$, $\hat{\mathbf{ y}}_1=[1,0,\cdots,0]^{\top}\in \mathbb{R}^J$. Similar the proof of Theorem \ref{thm3.4}, we only show that the matrix
\begin{align}\label{eq3.17}
\Phi\equiv [\Phi_1|\Phi_2]=\left[\begin{array}{c|c}\phi_1^{\top}&\varphi_1^{\top}\\
\vdots&\vdots\\
\phi_{I+J-3}^{\top}&\varphi_{I+J-3}^{\top}\\\end{array}\right]\in \mathbb{R}^{(I+J-3)\times (I+J-2)}
\end{align}
is of full row rank, where $\phi_{j}$ and $\varphi_j$ are given in \eqref{eq3.9} for $j\in \{1, \ldots, I+J-3\}$ which satisfy \eqref{eq3.11}. Let
\begin{align*}
\Theta\equiv [\Theta_1|\Theta_2]=\left[\begin{array}{c|c}\hat{y}_{2,1}\hat{\mathbf{x}}_{2,2}^{\top}&(\hat{\mathbf{y}}_{2,2}\otimes \hat{\mathbf{x}}_{2,2})^{\top}\\
\vdots&\vdots\\
\hat{y}_{R,1}\hat{\mathbf{x}}_{R,2}^{\top}&(\hat{\mathbf{y}}_{R,2}\otimes \hat{\mathbf{x}}_{R,2})^{\top}\\\end{array}\right]\in \mathbb{R}^{(R-1)\times (I-1)J},
\end{align*}
where $\Theta_2\in  \mathbb{R}^{(R-1)\times (I-1)(J-1)}$
Then \eqref{eq3.11} can be rewritten as
\begin{align}\label{eq3.19}
\Theta\left[\begin{array}{c}\phi_j\\\hline{\rm vec}(\widehat{\mathsf{U}}_j)\end{array}\right]=-\left[\begin{array}{c}\hat{x}_{2,1}(\varphi_j^{\top}\hat{\mathbf{y}}_{2,2})\\ \vdots\\\hat{x}_{R,1}(\varphi_j^{\top}\hat{\mathbf{y}}_{R,2})\end{array}\right],
\end{align}
for $j\in \{1, \ldots, I+J-3\}$. Using the fact that $(I-1)(J-1)\leqslant R-1= R^*=(I-1)(J-1)+1\leqslant (I-1)J$ (because $I\geqslant 2$), it follows from Lemma \ref{lem2.2} $(i)$ that ${\rm rank}(\Theta)= R-1$ and ${\rm rank}(\Theta_2)=(I-1)(J-1)= R-2$. Hence, there exists a column vector of $\Theta_1$, say $(I-1)$th column, such that $\widehat{\Theta}_2=[\Theta(:,I-1)|\Theta_2]$ is invertible.  From \eqref{eq3.19} and Lemma \ref{lem3.3}, we can obtain that the matrix $[\Phi_1(:,1:I-2)|\Phi_2]\in \mathbb{R}^{(I+J-3)\times (I+J-3)}$ is invertible, generically, where $\Phi_1$ and $\Phi_2$ are given in \eqref{eq3.17}. Hence, $\Phi$ is of full row rank. By implicit function theorem, there is a real solution curve containing the point $(\hat{\mathbf{x}}_1,\hat{\mathbf{y}}_1)$, i.e., there are infinitely many real solutions, $(\tilde{\mathbf{x}}_1,\tilde{\mathbf{y}}_1)$, nearby  $(\hat{\mathbf{x}}_1,\hat{\mathbf{y}}_1)$. Since $%
\widehat{Y}\odot \widehat{X}$ is of full column rank, we obtain that there  are infinitely many $\widetilde{X}=[\tilde{\mathbf{x}}_1,\hat{\mathbf{x}}_2,\cdots, \hat{\mathbf{
x}}_R]$ and $\widetilde{Y}=[\tilde{\mathbf{y}}_1,\hat{\mathbf{y}}_2,\cdots, \hat{\mathbf{ y}}_R]$ such that $%
\widetilde{Y}\odot \widetilde{X}$ is of full column rank. The proof is completed.
\end{proof}

%
%\begin{Theorem}
%\label{thm3.6} Suppose that $\mathcal{A}\in \mathbb{R}^{I\times J\times K}$
%has a CP decomposition as in \eqref{3tensor} with $R<R^*$. Then $\mathcal{A}$
%has a unique CP decomposition generically.
%\end{Theorem}

\subsection{Computing CP decompositions by homotopy method}

\label{sec3.2}

As discussed in subsection \ref{sec3.1}, computing CP decomposition of a
tensor $\mathcal{A}\in \mathbb{R}^{I\times J\times K}$ with $\mathrm{rank }(%
\mathcal{A})=R\leqslant \mathrm{min}\{K,R^*\}$ is equivalent to solving %
\eqref{eq3.8}. Since \eqref{eq3.8} is a polynomial system, we consider to
use a homotopy continuation method to solve it numerically.

The basic idea of using homotopy continuation method to solve a general
polynomial system $P(\mathbf{x})=\mathbf{0}$ is to deform $P(\mathbf{x})$ to another polynomial system $Q(\mathbf{x})$ whose solutions are
known. Under certain conditions, a smooth curve that emanates from a
solution of $Q(\mathbf{x})=\mathbf{0}$ will lead to a solution of $P(\mathbf{x})=\mathbf{0}$ by the homotopy%
\begin{equation*}
H(\mathbf{x},t)=\left( 1-t\right) \gamma Q\left( \mathbf{x}\right) +tP(\mathbf{x})=\mathbf{0}\text{, }t\in \left[
0,1\right] \text{,}
\end{equation*}%
where $\gamma $ is a generic nonzero complex number. If $Q\left( \mathbf{x}\right) $
is chosen properly, the following properties hold:
\begin{itemize}
\item[] \hspace{-8mm} $\bullet$ \textbf{Property0} (triviality): The solutions
of $Q(\mathbf{x})=\mathbf{0}$ are known.

\item[] \hspace{-8mm} $\bullet$ \textbf{Property1} (smoothness): The solution
set of $H(\mathbf{\mathbf{x}},t)=\mathbf{0}$ for $0\leq t\leq 1$ consists of a finite number of smooth
paths, each parameterized by $t \in [0,1]$.

\item[] \hspace{-8mm} $\bullet$ \textbf{Property2} (accessibility): Every
isolated solution of $H(\mathbf{x},1)=P(\mathbf{x})=\mathbf{0} $ can be reached by some path
originating at $t=0$.
\end{itemize}
To construct an appropriate homotopy for solving \eqref{eq3.8}, the multi-homogeneous B\'{e}zout's
number will be used. For a polynomial system $P(\mathbf{x})=(p_1(x),\ldots,p_n(x))$ where $\mathbf{x}=(x_{1},\ldots ,x_{n})$, the variables $
x_{1},\ldots ,x_{n}$ are partitioned into $m$ groups $z_{1}=\left(
x_{1}^{\left( 1\right) },\cdots ,x_{k_{1}}^{\left( 1\right) }\right) $, $%
z_{2}=\left( x_{1}^{\left( 2\right) },\cdots ,x_{k_{2}}^{\left( 2\right)
}\right) ,...,z_{m}=\left( x_{1}^{\left( m\right) },\cdots
,x_{k_{m}}^{\left( m\right) }\right) $ with $k_{1}+\cdots +k_{m}=n$. Let $%
d_{ij}$ be the degree of $p_{i}$ with respect to $z_{j}$ for $i=1,\ldots ,n$
and $j=1,\ldots ,m$. Then the multi-homogeneous B\'{e}zout's number of $P(\mathbf{x})$
with respect to $\left( z_{1},\ldots ,z_{m}\right) $ is the coefficient of $%
\lambda _{1}^{k_{1}}\lambda _{2}^{k_{2}}\cdots \lambda _{m}^{k_{m}}$ in the
product%
\begin{equation*}
\prod_{i=1}^{n}\left( d_{i1}\lambda _{1}+\cdots +d_{im}\lambda _{m}\right)
\text{.}
\end{equation*}

The following theorem plays a role in constructing a proper homotopy.

\begin{Theorem} \cite{Sommese:2005}
\label{thm3.7} Let $Q(\mathbf{x})$ be a system of polynomial chosen to have the same
multi-homogeneous form as $P(\mathbf{x})$ with respect to certain partition of the
variables $\mathbf{x}=\left( x_{1},\cdots ,x_{n}\right) $. Assume $Q(\mathbf{x})=\mathbf{0}$ has exactly
the multi-homogeneous B\'{e}zout's number of isolated solutions with
respect to this partition, and let
\begin{equation*}
H(\mathbf{x},t)=(1-t)\gamma Q(\mathbf{x})+tP(\mathbf{x})=\mathbf{0}
\end{equation*}%
where $t\in \left[ 0,1\right] $ and $\gamma \in \mathbb{C\setminus }\left\{
0\right\} $. If $\gamma =re^{i\theta }$ for some positive $r$, then for all
but finitely many $\theta $, Properties 1 and 2 hold.
\end{Theorem}

\bigskip

For solving $P(\mathbf{x},\mathbf{y})=\mathbf{0}$, where $P(\mathbf{x},\mathbf{y})$ is defined in \eqref{eq3.8}, we consider the starting system

\begin{equation}
Q_{0}\left( \mathbf{x},\mathbf{y}\right) =\left\{
\begin{array}{l}
\left( \alpha _{1}^{\top}\mathbf{x}\right) \left( \beta _{1}^{\top}\mathbf{y}\right) \\
\left( \alpha _{2}^{\top}\mathbf{x}\right) \left( \beta _{2}^{\top}\mathbf{y}\right) \\
\text{ \ \ }\vdots \\
\left( \alpha _{I+J-2}^{\top}\mathbf{x}\right) \left( \beta _{I+J-2}^{\top}\mathbf{y}\right) \\
\mathbf{c}_{x}^{\top }\mathbf{x}-1 \\
\mathbf{c}_{y}^{\top }\mathbf{y}-1%
\end{array}%
\right.  \label{eq4.2}
\end{equation}%
where any $I$ vectors of $\left\{ \alpha _{1},\ldots ,\alpha_{I+J-2}\right\} $ and any $J$ vectors of $\left\{ \beta _{1},\ldots ,\beta _{I+J-2}\right\} $ are linearly independent.

\begin{Theorem}
\label{thm3.8} Let $P(\mathbf{x},\mathbf{y})$ and $Q_{0}(\mathbf{x},\mathbf{y})$ be defined as \eqref{eq3.8} and \eqref{eq4.2} respectively.
Then all the isolated zeros $(\mathbf{x},\mathbf{y})$ in $\mathbb{C}^{I+J}$
of $P(\mathbf{x},\mathbf{y})$ can be found by using the homotopy
\begin{equation}
H(\mathbf{x},\mathbf{y},t)=(1-t)\gamma Q_{0}\left( \mathbf{x},\mathbf{y}%
\right) +tP\left( \mathbf{x},\mathbf{y}\right) =\mathbf{0}\text{, }t\in \left[ 0,1%
\right]  \label{eq4.3}
\end{equation}%
for almost all $\gamma \in \mathbb{C\setminus }\left\{ 0\right\} $.
Moreover, $P(\mathbf{x},\mathbf{y})=\mathbf{0}$ has at most $M=\frac{\left( I-1+J-1\right) !}{\left(
I-1\right) !\left( J-1\right) !}$ isolated solutions.
\end{Theorem}

\begin{proof} It is sufficient to verify that $Q_{0}\left( \mathbf{x},%
\mathbf{y}\right) $ satisfies all the assumptions of Theorem \ref{thm3.7}. Partition
the variables $\left( \mathbf{x},\mathbf{y}\right) $ into two groups: $%
\left( \mathbf{x}\right) $ and $\left( \mathbf{y}\right) $, we can see that
each of the first $I+J-2$ equations in \eqref{eq3.8} and \eqref{eq4.2} has
degree $1$ in $\left( \mathbf{x}\right) $ and $\left( \mathbf{y}\right) $.
In addition, the equation $\mathbf{c}_{x}^{\top }\mathbf{x}-1$ has degree $1$
in $\left( \mathbf{x}\right) $ and degree $0$ in $\left( \mathbf{y}\right)$, while the equation $\mathbf{c}_{y}^{\top }\mathbf{y}-1$ has degree $0$ in $\left( \mathbf{x}\right) $ and degree $1$ in $\left( \mathbf{y}\right)$.
Hence, $P$ and $Q_{0}$ have the same multi-homogeneous B\'{e}zout's
number and the number is the coefficient of $\lambda _{1}^{I}\lambda
_{2}^{J} $ in the polynomial $\left( 1\cdot \lambda _{1}+1\cdot \lambda
_{2}\right) ^{I+J-2}\left( 1\cdot \lambda _{1}+0\cdot \lambda _{2}\right)
\left( 0\cdot \lambda _{1}+1\cdot \lambda _{2}\right) $. By the direct calculation, we have this coefficient is equal to $M=\frac{\left( I-1+J-1\right) !}{%
\left( I-1\right) !\left( J-1\right) !}$, which is the number of isolated solutions of $Q_{0}(\mathbf{x},\mathbf{y})=\mathbf{0}$. The accessibility property implies that the number of isolated solutions of $P(\mathbf{x},\mathbf{y})=\mathbf{0}$ is at most $M$.
\end{proof}

\begin{Remark}\label{rem3.3}
From Theorems \ref{thm3.4} and \ref{thm3.8}, we have $M\geqslant s\geqslant R$, where  $R={\rm rank}(\mathcal{A})$ and $s$ is the number of vectors in $S_\mathbb{R}$ defined in \eqref{realsolset}. Assume that $R=R^{*}(=IJ-I-J+2)$. We obtain that  $M=R^*=R$ if $I\in \{1,2\}$ or $J\in \{1,2\}$. Hence, we have following two results.
\begin{itemize}
\item[(i)]  If ${\rm rank}(\mathcal{A})=R^*$ and $I\in \{1,2\}$ or $J\in \{1,2\}$, then  $s=R$ and hence $\mathcal{A}$ has a unique CP decomposition by Theorem \ref{thm3.5}.
\item[(ii)] When $I=1$ (i.e., $\mathcal{A}$ is a matrix), the critical number $R^*=1$. By Theorem \ref{thm_under} and assertion $(i)$, we conclude that  the CP decomposition of a matrix is unique only when $R = 1$.
\end{itemize}
\end{Remark}

Theorem \ref{thm3.8} suggests us that the polynomial system $P(\mathbf{x},\mathbf{y})=\mathbf{0}$ can be solved by using homotopy continuation method with homotopy in \eqref{eq4.3}. The solutions of the starting system $Q_{0}(\mathbf{x},\mathbf{y})=\mathbf{0}$ are easily obtained because of the linear product form. In $%
Q_{0}(\mathbf{x},\mathbf{y})$, any $I-1$ linear forms in $\mathbf{x}$
together with the equation $\mathbf{c}_{x}^{\top }\mathbf{x}-1$ form an $I\times I$ nonsingular linear system in $\mathbf{x}$ :
\begin{equation*}
\left[
\begin{array}{ccccc}
\alpha _{i_{1}} & \alpha _{i_{2}} & \cdots & \alpha _{i_{I-1}} & \mathbf{c}%
_{x}%
\end{array}%
\right] ^{\top}\mathbf{x}=\left[
\begin{array}{ccccc}
0 & 0 & \cdots & 0 & 1%
\end{array}%
\right] ^{\top},
\end{equation*}%
and the remaining $J-1$ linear forms in $\mathbf{y}$ together with $\mathbf{c}_{y}^{\top }\mathbf{y}-1$ form a $J\times J$ nonsingular linear system in $\mathbf{y}$. The solutions of these two linear systems determine a solution
of $Q_{0}(\mathbf{x},\mathbf{y})=\mathbf{0}$. Therefore, the starting system has
exactly $M$ solutions. The typical strategy to obtain the solution of $P(\mathbf{x},\mathbf{y})=\mathbf{0}$ is using the prediction-correction method: Let $\left( \mathbf{x}_{0},\mathbf{y}_{0}\right) $ be a solution of $Q_{0}\left(\mathbf{x},\mathbf{y}\right) =H(\mathbf{x},\mathbf{y},0)=\mathbf{0}$, and let $t_{0}=0 $.

\textbf{Prediction step}: Compute the tangent vector $\frac{d\left( \mathbf{x%
},\mathbf{y}\right) }{dt}$ to $H(\mathbf{x},\mathbf{y},t)=\mathbf{0}$ at $t_{k}$ by
solving the linear system%
\begin{equation*}
\frac{dH}{d\left( \mathbf{x},\mathbf{y}\right) }(\mathbf{x}_{k},\mathbf{y}%
_{k},t_{k})\frac{d\left( \mathbf{x},\mathbf{y}\right) }{dt}=-\frac{dH}{dt}(%
\mathbf{x}_{k},\mathbf{y}_{k},t_{k})\text{ for }\frac{d\left( \mathbf{x},%
\mathbf{y}\right) }{dt}\text{.}
\end{equation*}%
Then compute the approximate solution $(\mathbf{\tilde{x}},\mathbf{\tilde{y}}%
)$ to $(\mathbf{x}_{k+1},\mathbf{y}_{k+1})$ by
\begin{equation*}
(\mathbf{\tilde{x}},\mathbf{\tilde{y}})=(\mathbf{x}_{k},\mathbf{y}_{k})+h%
\frac{d\left( \mathbf{x},\mathbf{y}\right) }{dt}\text{, \ \ }t_{k+1}=t_{k}+h
\end{equation*}%
where $h$ is the stepsize.

\textbf{Correction step}: Use Newton's iterations. Initialize $(\mathbf{x}%
^{\left( 0\right) },\mathbf{y}^{\left( 0\right) })=(\mathbf{\tilde{x}},%
\mathbf{\tilde{y}})$. For $i=0,1,2,\cdots $, compute
\begin{equation*}
(\mathbf{x}^{\left( i+1\right) },\mathbf{y}^{\left( i+1\right) })=(\mathbf{x}%
^{\left( i\right) },\mathbf{y}^{\left( i\right) })-\left[ \frac{dH}{d\left(
\mathbf{x},\mathbf{y}\right) }(\mathbf{x}^{\left( i\right) },\mathbf{y}%
^{\left( i\right) },t_{k+1})\right] ^{-1}H(\mathbf{x}^{\left( i\right) },%
\mathbf{y}^{\left( i\right) },t_{k+1})
\end{equation*}%
until $\left\Vert H(\mathbf{x}^{\left( N\right) },\mathbf{y}^{\left(
N\right) },t_{k+1})\right\Vert $ is smaller than a given tolerance enough.
Then let $(\mathbf{x}_{k+1},\mathbf{y}_{k+1})=(\mathbf{x}^{\left( N\right) },%
\mathbf{y}^{\left( N\right) })$.

The prediction-correction iteration terminates until $t_{k_{0}}=1$ for some $%
k_{0}\in \mathbb{N}$. At this step, $(\mathbf{x}_{k_{0}},\mathbf{y}_{k_{0}})$
is the solution of $P(\mathbf{x},\mathbf{y})=\mathbf{0}$. The algorithm for computing the CP decomposition is summarized in Algorithm 1.

\begin{center}
\renewcommand{\arraystretch}{1.36}
\textbf{Algorithm~1}
\begin{tabular}[c]{|llll|}\hline  &&& \\
   & \multicolumn{1}{l}{\textsf{Input}:} &  \multicolumn{2}{l}{$\mathcal{A}\in \mathbb{R}^{I\times J\times K}$ with rank$\left(
\mathcal{A}\right) =R\leqslant \min \left\{ K,R^{\ast }\right\}$.} \\
   & \multicolumn{1}{l}{\textsf{Output}:} & \multicolumn{2}{l}{Factor matrices $\widehat{X}\in \mathbb{R}^{I\times R}$, $\widehat{Y}\in \mathbb{R}%
^{J\times R}$, and $\widehat{Z}\in \mathbb{R}^{K\times R}$.}  \\
   &&1.& \textsf{Form} the matricization $T\in \mathbb{R}^{IJ\times K}$ of $\mathcal{%
A}$ as \eqref{T}.\\
&&2.&  \textsf{Compute} the full rank factorization $T=EF^{\top }$, where $E\in
\mathbb{R}^{IJ\times R}$\\
&& &  has orthonormal columns and $F\in \mathbb{R}%
^{K\times R}$ has full column rank.\\
&&3.&  \textsf{Compute} a basis $\left\{ \mathbf{u}_{1},\mathbf{u}_{2},\cdots ,%
\mathbf{u}_{IJ-R}\right\} $ of nullspace $\mathcal{N}(E^{\top })$ \\
&&&    and
construct the square polynomial system \eqref{eq3.8}.\\
&&4.&  \textsf{Find} the real solution set $S_\mathbb{R}$ of system \eqref{eq3.8}.\\
&&5.&  \textsf{Find} the $R$ real solutions $\left\{ (\hat{\mathbf{x}}_{i},\hat{\mathbf{y}}%
_{i})\right\} _{i=1}^{R}\subseteq S_\mathbb{R}$ which satisfy  \\
&&&   $q_i(\mathbf{x}%
,\mathbf{y})=0$ (defined in \eqref{drop}) for $i=I+J-1,\cdots,IJ-R$.\\
&&6.&  \textsf{Stack} $\left\{ (\hat{\mathbf{x}}_{i},\hat{\mathbf{y}}_{i})\right\} _{i=1}^{R}$
in factor matrices $\widehat{X}\in \mathbb{R}^{I\times R}$and $\widehat{Y}\in \mathbb{R}%
^{J\times R}$:\\
&&&$\widehat{X} =[\hat{\mathbf{x}}_{1},\hat{\mathbf{x}}_{2},\cdots,\hat{\mathbf{x}}_{R}],
 \ \ \widehat{Y} =[\hat{\mathbf{y}}_{1},  \hat{\mathbf{y}}_{2},  \cdots,  \hat{\mathbf{y}}_{R}].$\\
&&7.&  $W=E^{\top}\cdot ( \widehat{X}\odot \widehat{Y}) \in \mathbb{R}^{R\times R}$ and $\widehat{Z}=F\cdot W^{-\top}\in \mathbb{R}^{K\times R}$.\\
   \hline
\end{tabular}
\end{center}

\textbf{Example 1}: Consider a $3\times 3\times 6$ tensor $\mathcal{A}$
generated by
\begin{equation}
\mathcal{A}=\sum_{r=1}^{4}\mathbf{x}_{r}\circ \mathbf{y}%
_{r}\circ \mathbf{z}_{r},  \label{ex1}
\end{equation}%
where
\begin{eqnarray*}
X &=&[\mathbf{x}_{1},\mathbf{x}_{2},\mathbf{x}_{3},%
\mathbf{x}_{4}]=\left[
\begin{array}{rrrr}
0 & 1 & \frac{1}{2} & \frac{1}{3} \\
1 & 0 & 1 & \frac{1}{2} \\
\frac{1}{2} & 1 & 0 & 1%
\end{array}%
\right] \in \mathbb{R}^{3\times 4}, \\
Y &=&[\mathbf{y}_{1},\mathbf{y}_{2},\mathbf{y}_{3},%
\mathbf{y}_{4}]=\left[
\begin{array}{rrrr}
\frac{1}{2} & 1 & 0 & 1 \\
1 & 0 & 1 & \frac{1}{2} \\
0 & 1 & \frac{1}{2} & -\frac{1}{3}%
\end{array}%
\right] \in \mathbb{R}^{3\times 4},\text{ and} \\
Z &=&[\mathbf{z}_{1},\mathbf{z}_{2},\mathbf{z}_{3},%
\mathbf{z}_{4}]=\left[
\begin{array}{rrrr}
1 & 1 & 1 & 1 \\
-1 & 1 & 1 & 1 \\
1 & -1 & 1 & 1 \\
1 & 1 & -1 & 1 \\
1 & 1 & 1 & -1 \\
-1 & -1 & 1 & 1%
\end{array}%
\right] \in \mathbb{R}^{6\times 4}\text{.}
\end{eqnarray*}%
The rank of tensor $\mathcal{A}$ is $4$, smaller than $\min \left\{ R^{\ast
},K\right\} =\min \left\{ 5,6\right\} $. We compute the CP decomposition by
\textbf{Algorithm 1}. In the first two steps, we have the matricization whose rank is
$4$ and whose full rank factorization $T=EF^{\top}$, where $E\in \mathbb{R}%
^{9\times 4}$ and $F\in \mathbb{R}^{6\times 4}$. In the third step, a basis
of nullspace of $E^{\top}$ can be chosen as $\left\{ \mathbf{u}_{1},\mathbf{u}%
_{2},\mathbf{u}_{3},\mathbf{u}_{4},\mathbf{u}_{5}\right\} $, where
\begin{eqnarray*}
\mathbf{u}_{1}^{\top} &=&\left[
\begin{array}{rrrrrrrrr}
13 & 4 & -8 & 0 & 0 & 0 & 0 & 0 & -5%
\end{array}%
\right] ^{\top}, \\
\mathbf{u}_{2}^{\top} &=&\left[
\begin{array}{rrrrrrrrr}
4 & 2 & -4 & 5 & 0 & 0 & 0 & -5 & 0%
\end{array}%
\right] ^{\top}, \\
\mathbf{u}_{3}^{\top} &=&\left[
\begin{array}{rrrrrrrrr}
68 & 19 & -38 & 15 & 0 & 0 & -30 & 0 & 0%
\end{array}%
\right] ^{\top}, \\
\mathbf{u}_{4}^{\top} &=&\left[
\begin{array}{rrrrrrrrr}
0 & 1 & 0 & 0 & 0 & -1 & 0 & 0 & 0%
\end{array}%
\right] ^{\top}, \\
\mathbf{u}_{5}^{\top} &=&\left[
\begin{array}{rrrrrrrrr}
26 & 33 & -26 & 20 & -10 & 0 & 0 & 0 & 0%
\end{array}%
\right] ^{\top}.
\end{eqnarray*}%
In step 4, the square polynomial system generated by the basis is
\begin{equation*}
P\left(\mathbf{x},\mathbf{y}\right) =\left\{
\begin{array}{l}
13y_{1}x_{1}+4y_{1}x_{2}-8y_{1}x_{3}-5y_{3}x_{3}=0, \\
4y_{1}x_{1}+2y_{1}x_{2}-4y_{1}x_{3}+5y_{2}x_{1}-5y_{3}x_{2}=0, \\
68y_{1}x_{1}+19y_{1}x_{2}-38y_{1}x_{3}+15y_{2}x_{1}-30y_{3}x_{1}=0, \\
y_{1}x_{2}-y_{2}x_{3}=0, \\
x_{1}+x_{2}+x_{3}-1=0, \\
y_{1}+y_{2}+y_{3}-1=0\text{.}%
\end{array}%
\right.
\end{equation*}%
where $(\mathbf{x},\mathbf{y})=(x_{1},x_{2},x_{3},y_{1},y_{2},y_{3})$.
This system has $6$ real solutions%
\begin{equation*}
\begin{array}{l}
\mathcal{S}_{\mathbb{R}}=\{\left( 0,\frac{2}{3},\frac{1}{3},\frac{1}{3},\frac{2}{3}%
,0\right) ,\left( \frac{1}{2},0,\frac{1}{2},\frac{1}{2},0,%
\frac{1}{2}\right) ,\left( \frac{1}{3},\frac{2}{3},0,0,\frac{2%
}{3},\frac{1}{3}\right) , \\
\text{ \ \ \ \  }\left( \frac{2}{11},\frac{3}{11},\frac{6}{11},\frac{%
6}{7},\frac{3}{7},-\frac{2}{7}\right) ,\left( \frac{19}{45},%
\frac{26}{135},\frac{52}{135},\frac{20}{63},\frac{10}{63},\frac{11}{21}%
\right) ,\left( 0,1,0,0,1,0\right) \}\text{.}%
\end{array}%
\end{equation*}%
In step 5, only first four solutions of $\mathcal{S}_{\mathbb{R}}$ satisfy the equation $%
26y_{1}x_{1}+33y_{1}x_{2}-26y_{1}x_{3}+20y_{2}x_{1}-10y_{2}x_{2}=0$, which
is generated from $\mathbf{u}_{5}$. In step 6, stack these solutions in
matrix factors%
\begin{eqnarray*}
\widehat{X} &=&[\hat{\mathbf{x}}_{1},\hat{\mathbf{x}}_{2},\hat{\mathbf{x}}_{3},\hat{\mathbf{x}}_{4}]=\left[
\begin{array}{rrrr}
0 & \frac{1}{2} & \frac{1}{3} & \frac{2}{11} \\
\frac{2}{3} & 0 & \frac{2}{3} & \frac{3}{11} \\
\frac{1}{3} & \frac{1}{2} & 0 & \frac{6}{11}%
\end{array}%
\right] \text{ } \\
\text{and }\widehat{Y} &=&[\hat{\mathbf{y}}_{1},\hat{\mathbf{y}}_{2},\hat{\mathbf{y}}_{3},\hat{\mathbf{y}}%
_{4}]=\left[
\begin{array}{rrrr}
\frac{1}{3} & \frac{1}{2} & 0 & \frac{6}{7} \\
\frac{2}{3} & 0 & \frac{2}{3} & \frac{3}{7} \\
0 & \frac{1}{2} & \frac{1}{3} & -\frac{2}{7}%
\end{array}%
\right] \text{.}
\end{eqnarray*}%
Following step 7 and 8, we have matrix factor
\begin{equation*}
\widehat{Z}=[\hat{\mathbf{z}}_{1},\hat{\mathbf{z}}_{2},\hat{\mathbf{z}}_{3},\hat{\mathbf{z}}_{4}]=\left[
\begin{array}{rrrr}
\frac{9}{4} & 4 & \frac{9}{4} & \frac{77}{36} \\
-\frac{9}{4} & 4 & \frac{9}{4} & \frac{77}{36} \\
\frac{9}{4} & -4 & \frac{9}{4} & \frac{77}{36} \\
\frac{9}{4} & 4 & -\frac{9}{4} & \frac{77}{36} \\
\frac{9}{4} & 4 & \frac{9}{4} & -\frac{77}{36} \\
-\frac{9}{4} & -4 & \frac{9}{4} & \frac{77}{36}%
\end{array}%
\right] \text{.}
\end{equation*}%
Note that $\hat{\mathbf{x}}_{1}=\frac{2}{3}\mathbf{x}_{1}$, $\hat{\mathbf{y}}_{1}=%
\frac{2}{3}\mathbf{y}_{1}$, and $\hat{\mathbf{z}}_{1}=\frac{9}{4}\mathbf{z}_{1}$.
Hence, rank one component $\hat{\mathbf{x}}_{1}\circ \hat{\mathbf{y}}%
_{1}\circ \hat{\mathbf{z}}_{1}=\mathbf{x}_{1}\circ \mathbf{y}_{1}\circ
\mathbf{z}_{1}$. Similarly, we have $\hat{\mathbf{x}}_{r}\circ \hat{\mathbf{y}}%
_{r}\circ \hat{\mathbf{z}}_{r}=\mathbf{x}_{r}\circ \mathbf{y}%
_{r}\circ \mathbf{z}_{r}$ for $r=1,2,3,4$. Therefore, the CP factorization of
tensor $\mathcal{A}$ is unique.

\begin{Remark}
\label{rem3.4} When tensor $\widetilde{\mathcal{A}}\in \mathbb{R}^{I\times
J\times K}$ may only be known with noise, i.e., $\widetilde{\mathcal{A}}=%
\mathcal{A}+\theta \frac{\mathcal{N}}{\|\mathcal{N}\|_F}$, where $R=\mathrm{%
rank}(\mathcal{A})$ and $\theta$ is the noise level, we hope to compute an
approximating CP decomposition of tensor $\widetilde{\mathcal{A}}$. Some
comments concerning the particular implementation of \textbf{Algorithm 1} are as
following:

\begin{itemize}
\item[(i)] In step 2, the factorization may be obtained by truncating small
singular value terms in singular value decomposition.
\item[(ii)] In step 5, if $R<R^{\ast }$, then the real solution $(\hat{\mathbf{x}}%
_{i},\hat{\mathbf{y}}_{i})\in S_\mathbb{R}$ may not exactly satisfy the $R^{\ast
}-R$ equations, $q_{i}(\mathbf{x},\mathbf{y})=0$ for $i=I+J-1,\cdots ,%
IJ-R$, defined in \eqref{drop}. Suppose that the real solution set $S_\mathbb{R}$ in \eqref{realsolset} has only $s$
vectors. From Theorem \ref{thm3.8}, we obtain that  $s\leqslant M=\frac{\left( I-1+J-1\right) !}{\left(
I-1\right) !\left( J-1\right) !}$.
\begin{itemize}
\item[(a)] If $s<R$, then \textbf{Algorithm 1} fails to get the approximating CP
decomposition.
\item[(b)] If $s\geqslant R$, then let $\delta ((\mathbf{x},\mathbf{y}))=\sqrt{%
\sum_{i=I+J-1}^{IJ-R}(q_{i}(\mathbf{x},\mathbf{y}))^{2}}$ and suppose
that
\begin{equation*}
\delta ((\hat{\mathbf{x}}_{1},\hat{\mathbf{y}}_{1}))\leqslant \delta ((\hat{\mathbf{x}}_{2},%
\hat{\mathbf{y}}_{2}))\leqslant \cdots \leqslant \delta ((\hat{\mathbf{x}}_{s},\hat{\mathbf{y}}_{s}))
\end{equation*}%
where $S_\mathbb{R}=\left\{ (\hat{\mathbf{x}}_{i},\hat{\mathbf{y}}_{i})\right\}
_{i=1}^{s}$. It is natural to chose $R$ real solutions $\left\{ (\hat{\mathbf{x}}_{i},\hat{\mathbf{y}}_{i})\right\} _{i=1}^{R}\subseteq S_\mathbb{R}$ having the
smallest values $\left\{ \delta ((\hat{\mathbf{x}}_{i},\hat{\mathbf{y}}_{i}))\right\}
_{i=1}^{R}$.
\end{itemize}
\end{itemize}
\end{Remark}

\section{The fourth-order tensor}
\label{sec4}

Now, we consider the CP decomposition of a fourth-order tensor $\mathcal{A}%
\in \mathbb{R}^{I\times J\times K\times L}$ with $\mathrm{rank}(\mathcal{A}%
)=R\leqslant L$. Suppose that the CP of the tensor $\mathcal{A}\in \mathbb{R}%
^{I\times J\times K\times L}$ is given by
\begin{align}  \label{4tensor}
\mathcal{A}=\sum_{r=1}^{R}\mathbf{x}_{r}\circ \mathbf{y}%
_{r}\circ \mathbf{z}_{r}\circ \mathbf{v}_{r},
\end{align}
where  for each $r\in \{1,2,\cdots, R\}$, $%
\mathbf{x}_{r}$, $\mathbf{y}_{r}$, $\mathbf{z}_{r}$ and $\mathbf{v}_{r}$ are
generic. Let  $X=[\mathbf{x}_{1},\mathbf{x}_{2},\cdots ,\mathbf{x}_{R}]\in \mathbb{R}%
^{I\times R}$, $Y=[\mathbf{y}_{1},\mathbf{y}_{2},\cdots ,\mathbf{y}_{R}]\in
\mathbb{R}^{J\times R}$, $Z=[\mathbf{z}_{1},\mathbf{z}_{2},\cdots ,\mathbf{z}_{R}]\in \mathbb{R}%
^{K\times R}$ and $V=[\mathbf{v}_{1},\mathbf{v}_{2},\cdots ,\mathbf{v}_{R}]\in
\mathbb{R}^{L\times R}$ be the factor matrices of $\mathcal{A}$, then $X$, $Y$, $Z$ and $Y$ are generic. Let the matricization
\begin{align}\label{T4}
T=[\mathrm{vec}(\mathcal{A}_{:::1}),\mathrm{vec}(\mathcal{A}_{:::2}),\cdots ,%
\mathrm{vec}(\mathcal{A}_{:::L})],
\end{align}
where $\mathrm{vec}(\mathcal{A}_{:::\ell })\equiv \mathrm{vec}([\mathrm{vec}(%
\mathcal{A}_{::1,\ell }),\mathrm{vec}(\mathcal{A}_{::2,\ell }),\cdots ,%
\mathrm{vec}(\mathcal{A}_{::K,\ell })])$ for $\ell =1,2,\cdots ,L$ and $X=[\mathbf{x}_{1},\mathbf{x}_{2},\cdots ,\mathbf{x}_{R}]\in \mathbb{R}%
^{I\times R}$, $Y=[\mathbf{y}_{1},\mathbf{y}_{2},\cdots ,\mathbf{y}_{R}]\in
\mathbb{R}^{J\times R}$, $Z=[\mathbf{z}_{1},\mathbf{z}_{2},\cdots ,\mathbf{z}_{R}]\in \mathbb{R}%
^{K\times R}$ and $V=[\mathbf{v}_{1},\mathbf{v}_{2},\cdots ,\mathbf{v}_{R}]\in
\mathbb{R}^{L\times R}$.
%\begin{equation*}
%\begin{array}{ll}
%X=[\mathbf{x}_{1},\mathbf{x}_{2},\cdots ,\mathbf{x}_{R}]\in \mathbb{R}%
%^{I\times R}, & Y=[\mathbf{y}_{1},\mathbf{y}_{2},\cdots ,\mathbf{y}_{R}]\in
%\mathbb{R}^{J\times R}, \\
%Z=[\mathbf{z}_{1},\mathbf{z}_{2},\cdots ,\mathbf{z}_{R}]\in \mathbb{R}%
%^{K\times R}, & V=[\mathbf{v}_{1},\mathbf{v}_{2},\cdots ,\mathbf{v}_{R}]\in
%\mathbb{R}^{L\times R},
%\end{array}%
%\end{equation*}%
Then we have
\begin{align}  \label{vectansor4}
T=(Z\odot Y\odot X) V^{\top }\in \mathbb{R}^{IJK\times L}.
\end{align}

Consider a factorization of $T$ of the form
\begin{align}  \label{EFdec4}
T=EF^{\top},
\end{align}
where $E\in \mathbb{R}^{IJK\times R}$ and $F\in \mathbb{R}^{L\times R}$ are
of full column rank. From \eqref{vectansor4}, \eqref{EFdec4} and Lemma \ref{lem2.2}, we have
\begin{align}  \label{YX4}
Z\odot Y\odot X=EW,
\end{align}
for some nonsingular $W\in \mathbb{R}^{R\times R}$. Our goal is to find an
invertible matrix $W$ such that \eqref{YX4} holds.

Since $E\in \mathbb{R}^{IJK\times R}$ is of full
column rank, the dimension of null space of $E^{\top}$, $\mathcal{N}%
(E^{\top})=\{\mathbf{u}\in \mathbb{R}^{IJK}|E^{\top}\mathbf{u}=0\}$, is $%
IJK-R$. The following theorem is the generalization of Theorem \ref{thm3.1}.

\begin{Theorem}
\label{thm4.1} Let $E\in \mathbb{R}^{IJK\times R}$, $\tilde{\mathbf{x}}\in
\mathbb{R}^{I}$, $\tilde{\mathbf{y}}\in \mathbb{R}^{J}$ and $\tilde{\mathbf{z}}\in \mathbb{R}^{K}$. Then $\tilde{\mathbf{z}}\otimes \tilde{\mathbf{y}}
\otimes \tilde{\mathbf{x}}$ belongs to the column space of $E$ if and only
if $(\tilde{\mathbf{x}},\tilde{\mathbf{y}},\tilde{\mathbf{z}})$ is a
solution of the cubic equation $\mathbf{u}^{\top }\left( \mathbf{z}\otimes
\mathbf{y}\otimes \mathbf{x}\right) =0$, where $\mathbf{u}$ is any vector in
$\mathcal{N}(E^{\top })$.
\end{Theorem}

Let $\{\mathbf{u}_{1},\mathbf{u}_{2},\cdots ,\mathbf{u}_{IJK-R}\}$ be a
basis of $\mathcal{N}(E^{\top })$. We consider the system
\begin{align}  \label{eq4.6}
\left\{
\begin{array}{c}
\mathbf{u}_{1}^{\top }\left( \mathbf{z}\otimes \mathbf{y}\otimes \mathbf{x}%
\right) =0, \\
\vdots  \\
\mathbf{u}_{IJK-R}^{\top }\left( \mathbf{z}\otimes \mathbf{y}\otimes
\mathbf{x}\right) =0, \\
\mathbf{c}_{x}^{\top }\mathbf{x}=1, \\
\mathbf{c}_{y}^{\top }\mathbf{y}=1, \\
\mathbf{c}_{z}^{\top }\mathbf{z}=1, \\
\end{array}%
\right.
\end{align}
where $\mathbf{c}_{x}\in \mathbb{R}^{I}$, $\mathbf{c}_{y}\in \mathbb{R}^{J}$
and $\mathbf{c}_{z}\in \mathbb{R}^{K}$ are randomly generated. The system %
\eqref{eq4.6} has $IJK-R+3$ polynomial equations in $I+J+K$
unknowns, where $R=\mathrm{rank}(\mathcal{A})$. Define a critical number
\begin{align*}  %\label{Rstar4}
R^*=IJK-I-J-K+3.
\end{align*}
If $R=R^{\ast }$ then \eqref{eq4.6} has $I+J+K$ polynomial equations in $%
I+J+K$ unknowns. Note that the critical number $R^*$ is the same number defined in \eqref{Rstar} when we consider an $(I\times J\times K\times L)$ tensor.

%When a fourth-order tensor $\mathcal{A}\in \mathbb{R}^{I\times J\times
%K\times L}$ has a CP decomposition as in \eqref{4tensor} with $R^{\ast
%}<R\leqslant L$. The system \eqref{eq4.6} is an underdetermined system.
%Hence, \eqref{eq4.6} has infinitely many real solutions. Similar conclusion
%of Theorem \ref{thm_under} can be obtained, i.e., $\mathcal{A}$ has
%infinitely many CP decompositions generically.

When a fourth-order tensor $\mathcal{A}\in \mathbb{R}^{I\times J\times
K\times L}$ has a CP decomposition as in \eqref{4tensor} with $R\leqslant
\min \{L,R^{\ast }\}$. Let
\begin{align}  \label{eq4.8}
P(\mathbf{x},\mathbf{y},\mathbf{z})=\left[
\begin{array}{c}
p_{1}(\mathbf{x},\mathbf{y},\mathbf{z}) \\
\vdots  \\
p_{I+J+K-3}(\mathbf{x},\mathbf{y},\mathbf{z}) \\
p_{I+J+K-2}(\mathbf{x},\mathbf{y},\mathbf{z}) \\
p_{I+J+K-1}(\mathbf{x},\mathbf{y},\mathbf{z}) \\
p_{I+J+K}(\mathbf{x},\mathbf{y},\mathbf{z}) \\
\end{array}%
\right] \equiv \left[
\begin{array}{c}
\mathbf{u}_{1}^{\top }\left( \mathbf{z}\otimes \mathbf{y}\otimes \mathbf{x}%
\right)  \\
\vdots  \\
\mathbf{u}_{I+J+K-3}^{\top }\left( \mathbf{z}\otimes \mathbf{y}\otimes
\mathbf{x}\right)  \\
\mathbf{c}_{x}^{\top }\mathbf{x}-1 \\
\mathbf{c}_{y}^{\top }\mathbf{y}-1 \\
\mathbf{c}_{z}^{\top }\mathbf{z}-1%
\end{array}%
\right] .
\end{align}
If $R<R^{\ast }$, the system $P(\mathbf{x},\mathbf{y},\mathbf{z})=\mathbf{0}$ is
obtained by dropping $R^{\ast }-R$ equations,
\begin{align}\label{drop4}
q_{i}(\mathbf{x},\mathbf{y},\mathbf{z})=\mathbf{u}_{i}^{\top }\left(
\mathbf{z}\otimes \mathbf{y}\otimes \mathbf{x}\right) =0,\text{ for }%
i=I+J+K-2,\cdots ,IJK-R,
\end{align}
of system \eqref{eq4.6}. From \eqref{4tensor}, \eqref{YX4} and Theorem \ref%
{thm4.1}, we know that
\begin{align}  \label{scal_sol4}
(\hat{\mathbf{x}}_{r},\hat{\mathbf{y}}_{r})=\left( \frac{1}{\mathbf{c}%
_{x}^{\top }\mathbf{x}_{r}}\mathbf{x}_{r},\frac{1}{\mathbf{c}_{y}^{\top }%
\mathbf{y}_{r}}\mathbf{y}_{r},\frac{1}{\mathbf{c}_{z}^{\top }\mathbf{z}_{r}}%
\mathbf{z}_{r}\right) ,\text{ for }r=1,2,\cdots ,R,
\end{align}
are real solutions of $P(\mathbf{x},\mathbf{y},\mathbf{z})=\mathbf{0}$, where $\mathbf{x}%
_{r}$, $\mathbf{y}_{r}$ and $\mathbf{z}_{r}$ are given in \eqref{4tensor}.
Next, we will show that all real solutions of $P(\mathbf{x},\mathbf{y},%
\mathbf{z})=\mathbf{0}$ are isolated, generically. %The following lemma is useful in proving Theorem \ref{thm4.2}.

\begin{Theorem}
\label{thm4.2}  Suppose that $\mathcal{A}\in \mathbb{R}^{I\times J\times
K\times L}$ has a CP decomposition as in \eqref{4tensor} with $R\leqslant R^*$
and $P(\mathbf{x},\mathbf{y},\mathbf{z})$ has the form in \eqref{eq4.8},
where $\{\mathbf{u}_j\}_{j=1}^{I+J+K-3}$ is an arbitrary linearly independent
set of $\mathcal{N}(E^{\top})$ and $\mathbf{c}_x\in \mathbb{R}^I$, $\mathbf{c%
}_y\in \mathbb{R}^J$, $\mathbf{c}_z\in \mathbb{R}^K$ are randomly generated. Let $\{(\hat{\mathbf{x}}_r,\hat{\mathbf{y}}_r,\hat{\mathbf{z}}_r)\}_{r=1}^R$
be defined in \eqref{scal_sol4}.
Then for each $r\in \{1,\ldots, R\}$, $(\hat{\mathbf{x}}_r,\hat{\mathbf{y}}_r,\hat{\mathbf{z}}_r)$ is
an isolated solution of $P(\mathbf{x},\mathbf{y},\mathbf{z})=\mathbf{0}$, generically.
%
%\begin{itemize}
%\item[(i).] $(\hat{\mathbf{x}}_r,\hat{\mathbf{y}}_r,\hat{\mathbf{z}}_r)$ is
%an isolated solution of $P(\mathbf{x},\mathbf{y},\mathbf{z})=0$, generically;
%
%\item[(ii).] if $(\hat{\mathbf{x}},\hat{\mathbf{y}},\hat{\mathbf{z}})$ is a
%real solution of $P(\mathbf{x},\mathbf{y},\mathbf{z})=0$, then $(\hat{%
%\mathbf{x}},\hat{\mathbf{y}},\hat{\mathbf{z}})$ is isolated, generically.
%\end{itemize}
\end{Theorem}

\begin{proof}
Without loss of generality, we claim that $(\hat{\mathbf{x}}_1,\hat{\mathbf{y}}_1,\hat{\mathbf{z}}_1)$ is isolated. It suffices to show that the Jacobian matrix $DP(\hat{\mathbf{x}}_1,\hat{\mathbf{y}}_1,\hat{\mathbf{z}}_1)$ is invertible. For convenience, we may assume $\hat{\mathbf{ x}}_1=[1,0\cdots,0]^{\top}\in \mathbb{R}^I$, $\hat{\mathbf{ y}}_1=[1,0,\cdots,0]^{\top}\in \mathbb{R}^J$, $\hat{\mathbf{ z}}_1=[1,0,\cdots,0]^{\top}\in \mathbb{R}^K$ and the other $R-1$ solutions are $(\hat{\mathbf{ x}}_r,\hat{\mathbf{ y}}_r,\hat{\mathbf{ z}}_r)$ for $r=2,3,\ldots, R$. For each $j\in \{1,\cdots,I+J+K-3\}$, since $p_j(\hat{\mathbf{x}}_1,\hat{\mathbf{y}}_1,\hat{\mathbf{z}}_1)=0$, we have $\mathbf{u}_j(1)=0$, where $\mathbf{u}_j(1)$ is the first component of vector $\mathbf{u}_j$. Let $\mathcal{U}_j={\rm vec}^{-1}(\mathbf{u}_j)\in \mathbb{R}^{I\times J\times K}$ be a third-order tensor. Then $\mathcal{U}_j(1,1,1)=0$.
Denote
\begin{align}\label{partitionU}
&\left[\begin{array}{c}0\\\phi_j\end{array}\right]\equiv \mathcal{U}_j(:,1,1)\in \mathbb{R}^{I},\ \ \left[\begin{array}{c}0\\\varphi_j\end{array}\right]\equiv \mathcal{U}_j(1,:,1)\in \mathbb{R}^{J},\ \ \left[\begin{array}{c}0\\\xi_j\end{array}\right]\equiv \mathcal{U}_j(1,1,:)\in \mathbb{R}^{K},\nonumber\\
&\widehat{U}_{xy}\equiv \mathcal{U}_j(2:I,2:J,1)\in \mathbb{R}^{(I-1)\times (J-1)},\ \ \widehat{U}_{yz}\equiv \mathcal{U}_j(1,2:J,2:K)\in \mathbb{R}^{(J-1)\times (K-1)},\nonumber\\
&\widehat{U}_{xz}\equiv \mathcal{U}_j(2:I,1,2:K)\in \mathbb{R}^{(I-1)\times (K-1)},\\
&\widehat{\mathcal{U}}_j\equiv \mathcal{U}_j(2:I,2:J,2:K)\in \mathbb{R}^{(I-1)\times (J-1)\times (K-1)}.\nonumber
\end{align}
Then $Dp_j(\hat{\mathbf{x}}_1,\hat{\mathbf{y}}_1,\hat{\mathbf{z}}_1)=[0,\phi_j^{\top}|0,\varphi_j^{\top}|0,\xi_j^{\top}]$.
%\begin{align*}
%\left[\begin{array}{c}0\\\phi_j\end{array}\right]=D_xp_j(\hat{x},\hat{y})=\mathsf{U}_j^{\top}\hat{y}\in\mathbb{R}^{I},\ \ \left[\begin{array}{c}0\\\varphi_j\end{array}\right]=D_yp_j(\hat{x},\hat{y})=\mathsf{U}_j\hat{x}\in\mathbb{R}^{J}.
%\end{align*}
It follows from \eqref{eq4.8} that
\begin{align*}
DP(\hat{\mathbf{x}}_1,\hat{\mathbf{y}}_1,\hat{\mathbf{z}}_1)=\left[\begin{array}{cc|cc|cc}0&\phi_1^{\top}&0&\varphi_1^{\top}&0&\xi_1^{\top}\\
\vdots&\vdots&\vdots&\vdots\\
0&\phi_{I+J+K-3}^{\top}&0&\varphi_{I+J+K-3}^{\top}&0&\xi_{I+J+K-3}^{\top}\\\hline
&\mathbf{ c}_x^{\top} &&0&&0\\
&0 &&\mathbf{ c}_y^{\top}&&0\\
&0 &&0&&\mathbf{ c}_z^{\top}\\
\end{array}\right].
\end{align*}
Since $\mathbf{ c}_x\in \mathbb{R}^I$, $\mathbf{ c}_y\in \mathbb{R}^J$ and $\mathbf{ c}_z\in \mathbb{R}^K$are randomly generated, the Jacobian matrix $DP(\hat{\mathbf{x}}_1,\hat{\mathbf{y}}_1,\hat{\mathbf{z}}_1)$ is invertible if and only if the matrix
\begin{align}\label{eq4.11}
\Phi=\left[\begin{array}{c|c|c}\phi_1^{\top}&\varphi_1^{\top}&\xi_1^{\top}\\
\vdots&\vdots&\vdots\\
\phi_{I+J+K-3}^{\top}&\varphi_{I+J+K-3}^{\top}&\xi_{I+J+K-3}^{\top}\\\end{array}\right]
\end{align}
is invertible. In the following, we  show that $\Phi$ is invertible, generically, if $\{\mathbf{u}_j\}_{j=1}^{I+J+K-3}$ is an arbitrary linearly independent set of   $\mathcal{N}(E^{\top})$. Let
\begin{align*}
\hat{\mathbf{x}}_r=\left[\begin{array}{c}\hat{x}_{r,1}\\\hat{\mathbf{x}}_{r,2}\end{array}\right],\ \
\hat{\mathbf{y}}_r=\left[\begin{array}{c}\hat{y}_{r,1}\\\hat{\mathbf{y}}_{r,2}\end{array}\right],\ \
\hat{\mathbf{z}}_r=\left[\begin{array}{c}\hat{z}_{r,1}\\\hat{\mathbf{z}}_{r,2}\end{array}\right].
\end{align*}
Since $(\hat{\mathbf{x}}_r,\hat{\mathbf{y}}_r,\hat{\mathbf{z}}_r)$ for $r=2,3,\ldots, R$ are solutions of $P(\mathbf{x},\mathbf{y},\mathbf{z})=\mathbf{0}$ in \eqref{eq4.8}, we obtain that for each $j\in \{1,2,\cdots,I+J+K-3\}$, $\mathbf{u}_j$ in \eqref{partitionU} satisfies
\begin{align}\label{eq4.12}
0=~&\mathbf{u}_j^{\top}(\hat{\mathbf{x}}_r\otimes\hat{\mathbf{y}}_r\otimes\hat{\mathbf{z}}_r)\nonumber\\
=~&\hat{y}_{r,1}\hat{z}_{r,1}(\hat{\mathbf{x}}_{r,2}^{\top}\phi_j)+\hat{x}_{r,1}\hat{z}_{r,1}(\hat{\mathbf{y}}_{r,2}^{\top}\varphi_j)+\hat{x}_{r,1}\hat{y}_{r,1}(\hat{\mathbf{z}}_{r,2}^{\top}\xi_j)\nonumber\\
&+\hat{z}_{r,1}\hat{\mathbf{x}}_{r,2}^{\top}\widehat{U}_{xy}\hat{\mathbf{y}}_{r,2}+\hat{x}_{r,1}\hat{\mathbf{y}}_{r,2}^{\top}\widehat{U}_{yz}\hat{\mathbf{z}}_{r,2}+\hat{y}_{r,1}\hat{\mathbf{x}}_{r,2}^{\top}\widehat{U}_{xz}\hat{\mathbf{z}}_{r,2}\nonumber\\
&+(\hat{\mathbf{z}}_{r,2}\otimes\hat{\mathbf{y}}_{r,2}\otimes \hat{\mathbf{x}}_{r,2})^{\top}{\rm vec}(\widehat{\mathcal{U}}_j)\nonumber\\
=~&\hat{y}_{r,1}\hat{z}_{r,1}(\hat{\mathbf{x}}_{r,2}^{\top}\phi_j)+\hat{x}_{r,1}\hat{z}_{r,1}(\hat{\mathbf{y}}_{r,2}^{\top}\varphi_j)+\hat{x}_{r,1}\hat{y}_{r,1}(\hat{\mathbf{z}}_{r,2}^{\top}\xi_j)\nonumber\\
&+\hat{z}_{r,1}(\hat{\mathbf{y}}_{r,2}\otimes\hat{\mathbf{x}}_{r,2})^{\top}{\rm vec}(\widehat{U}_{xy})+\hat{x}_{r,1}(\hat{\mathbf{z}}_{r,2}\otimes\hat{\mathbf{y}}_{r,2})^{\top}{\rm vec}(\widehat{U}_{yz})\nonumber\\
&+\hat{y}_{r,1}(\hat{\mathbf{z}}_{r,2}\otimes\hat{\mathbf{x}}_{r,2})^{\top}{\rm vec}(\widehat{U}_{xz})+(\hat{\mathbf{z}}_{r,2}\otimes\hat{\mathbf{y}}_{r,2}\otimes \hat{\mathbf{x}}_{r,2})^{\top}{\rm vec}(\widehat{\mathcal{U}}_j),
\end{align}
for $r=2,3,\ldots,R$.
Since $\hat{\mathbf{ x}}_r\in \mathbb{R}^{I}$, $\hat{\mathbf{ y}}_r\in \mathbb{R}^J$ and $\hat{\mathbf{z}}_r\in \mathbb{R}^K$ are generic vectors, $\hat{\mathbf{x}}_{r,2}\in \mathbb{R}^{I-1}$, $\hat{\mathbf{y}}_{r,2}\in \mathbb{R}^{J-1}$ and $\hat{\mathbf{z}}_{r,2}\in \mathbb{R}^{K-1}$  are also generic vectors. Let
\begin{align*}
\Theta=\left[\begin{array}{rrr}\hat{z}_{2,1}(\hat{\mathbf{y}}_{2,2}\otimes \hat{\mathbf{x}}_{2,2})&\cdots&\hat{z}_{R,1}(\hat{\mathbf{y}}_{R,2}\otimes \hat{\mathbf{x}}_{R,2})\\
\hat{x}_{2,1}(\hat{\mathbf{z}}_{2,2}\otimes\hat{\mathbf{y}}_{2,2})&\cdots&\hat{x}_{R,1}(\hat{\mathbf{z}}_{R,2}\otimes\hat{\mathbf{y}}_{R,2})\\
\hat{y}_{2,1}(\hat{\mathbf{z}}_{2,2}\otimes\hat{\mathbf{x}}_{2,2})&\cdots&\hat{y}_{R,1}(\hat{\mathbf{z}}_{R,2}\otimes\hat{\mathbf{x}}_{R,2})\\
(\hat{\mathbf{z}}_{2,2}\otimes\hat{\mathbf{y}}_{2,2}\otimes \hat{\mathbf{x}}_{2,2})&\cdots&(\hat{\mathbf{z}}_{R,2}\otimes\hat{\mathbf{y}}_{R,2}\otimes \hat{\mathbf{x}}_{R,2})\\
\end{array}\right]^{\top},\\
 \mathbf{ b}_j=\left[\begin{array}{c}\hat{y}_{2,1}\hat{z}_{2,1}(\hat{\mathbf{x}}_{2,2}^{\top}\phi_j)+\hat{x}_{2,1}\hat{z}_{2,1}(\hat{\mathbf{y}}_{2,2}^{\top}\varphi_j)+\hat{x}_{2,1}\hat{y}_{2,1}(\hat{\mathbf{z}}_{2,2}^{\top}\xi_j)\\ \vdots\\\hat{y}_{R,1}\hat{z}_{R,1}(\hat{\mathbf{x}}_{R,2}^{\top}\phi_j)+\hat{x}_{R,1}\hat{z}_{R,1}(\hat{\mathbf{y}}_{R,2}^{\top}\varphi_j)+\hat{x}_{R,1}\hat{y}_{R,1}(\hat{\mathbf{z}}_{R,2}^{\top}\xi_j)\end{array}\right].
\end{align*}
Then $\Theta\in \mathbb{R}^{(R-1)\times (IJK-I-J-K+2)}$ and $ \mathbf{ b}_j\in \mathbb{R}^{R-1}$.
From Lemma \ref{lem2.2} $(ii)$ and using the fact that $R\leqslant R^*=IJK-I-J-K+3$, we obtain that ${\rm rank}(\Theta)= R-1$. It follows from \eqref{eq4.12} that  $\phi_j$, $\varphi_j$, $\xi_j$, $\widehat{U}_{xy}$,  $\widehat{U}_{yz}$,  $\widehat{U}_{xz}$  and $\widehat{\mathcal{U}}_j$ for $j=1,2,\cdots,I+J+K-3$ should satisfy
\begin{align}\label{eq4.14}
\Theta[{\rm vec}(\widehat{U}_{xy})^{\top}, {\rm vec}(\widehat{U}_{yz})^{\top},{\rm vec}(\widehat{U}_{xz})^{\top},{\rm vec}(\widehat{\mathcal{U}}_j)^{\top}]^{\top}=-\mathbf{ b}_j.
\end{align}
Since ${\rm rank}(\Theta)= R-1$, the linear system \eqref{eq4.14} has solution for arbitrary $\phi_j$, $\varphi_j$ and $\xi_j$.
 Since $\left\{\mathbf{u}_j\right\}_{j=1}^{I+J+K-3}$ is an arbitrary linearly independent set of   $\mathcal{N}(E^{\top})$, the matrix $\Phi$ in \eqref{eq4.11} is invertible generically. Hence, $(\hat{\mathbf{x}}_1,\hat{\mathbf{y}}_1,\hat{\mathbf{z}}_1)$ is isolated.
\end{proof}

When a fourth-order tensor $\mathcal{A}\in \mathbb{R}^{I\times J\times
K\times L}$ has a CP decomposition as in \eqref{4tensor} with $R^{\ast
}<R\leqslant L$. The system \eqref{eq4.6} is an underdetermined system.
Hence, \eqref{eq4.6} has infinitely many real solutions. Similar conclusion
of Theorem \ref{thm_under} can be obtained, i.e., $\mathcal{A}$ has
infinitely many CP decompositions generically.

The homotopy algorithm for computing the CP decomposition of a fourth-order tensor $\mathcal{A}\in \mathbb{R}^{I\times J\times K\times L}$ is summarized in Algorithm 2.

\begin{center}
\renewcommand{\arraystretch}{1.36}
\textbf{Algorithm~2} \vspace{0.1cm}
\begin{tabular}[c]{|llll|}\hline \vspace*{-0.3cm} &&& \\
   & \multicolumn{1}{l}{\textsf{Input}:} &  \multicolumn{2}{l}{$\mathcal{A}\in \mathbb{R}^{I\times J\times K\times L}$ with rank$\left(
\mathcal{A}\right) =R \leq \min \left\{ R^{\ast },L\right\} $.} \\
   & \multicolumn{1}{l}{\textsf{Output}:} & \multicolumn{2}{l}{Factor matrices $\widehat{X}\in \mathbb{R}^{I\times R}$, $\widehat{Y}\in \mathbb{R}%
^{J\times R}$, $\widehat{Z}\in \mathbb{R}^{K\times R}$ and $\widehat{V}\in \mathbb{R}^{L\times
R} $.}  \\
   &&1.& \textsf{Form} the matricization $T\in \mathbb{R}^{IJK\times L}$ of $\mathcal{A}$
as \eqref{T4}.\\
&&2.&  \textsf{Compute} the full rank factorization $T=EF^{\top }$, where $E\in
\mathbb{R}^{IJK\times R}$\\
&& &  has orthonormal columns and $F\in \mathbb{R}%
^{L\times R}$ has full column rank.\\
&&3.&  \textsf{Compute} a basis $\left\{ \mathbf{u}_{1},\mathbf{u}_{2},\cdots ,%
\mathbf{u}_{IJK-R}\right\} $ of nullspace $\mathcal{N}(E^{\top })$ \\
&&&    and
construct the square polynomial system \eqref{eq4.8}.\\
&&4.&  \textsf{Find} the real solutions of the system $P\left( \mathbf{x},\mathbf{y},\mathbf{z}\right)=\mathbf{0}$ in \eqref{eq4.8}.\\
&&5.&  \textsf{Find} the $R$ real solutions $\left\{ (\hat{\mathbf{x}}_{i},\hat{\mathbf{y}}_{i},\hat{\mathbf{z}}_{i})\right\} _{i=1}^{R}$ that satisfy  \eqref{drop4}.  \\
&&6.&  \textsf{Stack} $\left\{ (\hat{\mathbf{x}}_{i},\hat{\mathbf{y}}_{i},\hat{\mathbf{z}}_{i})\right\}
_{i=1}^{R}$
in factor matrices $\widehat{X} =[\hat{\mathbf{x}}_{1},\hat{\mathbf{x}}_{2},\cdots,\hat{\mathbf{x}}_{R}]$,\\
&&&
$\widehat{Y} =[\hat{\mathbf{y}}_{1},  \hat{\mathbf{y}}_{2},  \cdots,  \hat{\mathbf{y}}_{R}],\text{ and }\widehat{Z}=[\hat{\mathbf{z}}_{1},\hat{\mathbf{z}}_{2},\cdots,\hat{\mathbf{z}}_{R}].$\\
&&7.&  $W=E^{\top}\cdot ( \widehat{Z}\odot \widehat{Y}\odot \widehat{X}) \in \mathbb{R}^{R\times R}$ and $\widehat{V}=F\cdot W^{-\top}\in \mathbb{R}^{L\times R}$.\\
   \hline
\end{tabular}
\end{center}

In step 4, the square polynomial system $P\left( \mathbf{x},\mathbf{y},\mathbf{z}\right)=\mathbf{0}$ can be solved by the homotopy continuation method with homotopy
\begin{equation*}
H(\mathbf{x},\mathbf{y},\mathbf{z},t)=(1-t)\gamma Q_{0}\left( \mathbf{x},\mathbf{y},\mathbf{z}\right) +tP\left( \mathbf{x},\mathbf{y},\mathbf{z},\right) = \mathbf{0},
\end{equation*}
where the starting system
\begin{equation*}
Q_{0}\left( \mathbf{x},\mathbf{y},\mathbf{z}\right) =\left\{
\begin{array}{l}
\left( \alpha _{1}^{\top}\mathbf{x}\right) \left( \beta _{1}^{\top}\mathbf{y}%
\right) \left( \gamma _{1}^{\top}\mathbf{z}\right) \\
\left( \alpha _{2}^{\top}\mathbf{x}\right) \left( \beta _{2}^{\top}\mathbf{y}%
\right) \left( \gamma _{2}^{\top}\mathbf{z}\right) \\
\text{ \ \ }\vdots \\
\left( \alpha _{I+J+K-3}^{\top}\mathbf{x}\right) \left( \beta _{I+J+K-3}^{\top}%
\mathbf{y}\right) \left( \gamma _{I+J+K-3}^{\top}\mathbf{z}\right) \\
\mathbf{c}_{x}^{\top }\mathbf{x}-1 \\
\mathbf{c}_{y}^{\top }\mathbf{y}-1 \\
\mathbf{c}_{z}^{\top }\mathbf{z}-1\text{ }%
\end{array}%
\right.    %\label{eq5.3}
\end{equation*}%
in which any $I$ vectors of $\left\{ \alpha _{1},\ldots ,\alpha_{I+J+K-3}\right\} $, any $J$ vectors of $\left\{ \beta _{1},\ldots ,\beta _{I+J+K-3}\right\} $, and any $K$ vectors of $\left\{
\gamma _{1},\ldots ,\gamma _{I+J+K-3}\right\} $ are linearly
independent. Note that $P$ and $Q_0$ have the same multi-homogeneous B\'{e}zout's
number $M=\frac{\left( I-1+J-1+K-1\right) !}{\left(
I-1\right) !\left( J-1\right) ! \left( K-1\right) !}$, which is the upper bound of the number of isolated solutions of $P\left(\mathbf{x},\mathbf{y},\mathbf{z}\right) = \mathbf{0}$.

\bigskip

\section{Numerical computations and experiments}

In this section, we consider the numerical computation for the CP
decomposition. The test tensors are generated in the following way:%
\begin{equation}
\mathcal{\tilde{A}}=\mathcal{A}+\theta \frac{\mathcal{N}}{\left\Vert
\mathcal{N}\right\Vert_F }  \label{eq6.1}
\end{equation}%
in which $\theta $ denotes the noise level, and $\mathcal{A}$ has exact rank-%
$R$ CP decomposition in \eqref{decom}.
The entries of tensor $\mathcal{N}$ are randomly generated with standard
normal distribution $N(0,1)$.

For a small nonzero noise level, the matricization $\tilde{T}$ of $\mathcal{%
\tilde{A}}$ is generally of full rank, but it is close to a rank-$R$ matrix \cite{Lee:2009}.
In this case, we resume the rank-$R$ matrix by truncating small singular
value terms in singular value decomposition. Let the SVD of $\tilde{T}\in
\mathbb{R}^{m\times n}$ be
\begin{equation*}
\tilde{T}=U\Sigma V^{\top}=\sigma _{1}\mathbf{u}_{1}\mathbf{v}_{1}^{\top}+\sigma
_{2}\mathbf{u}_{2}\mathbf{v}_{2}^{\top}+\cdots +\sigma _{n}\mathbf{u}_{n}%
\mathbf{v}_{n}^{\top}\text{,}
\end{equation*}%
where $U=[
\mathbf{u}_{1},  \cdots ,  \mathbf{u}_{m}] $ and $V=[
\mathbf{v}_{1},  \cdots ,  \mathbf{v}_{n}] $ are orthonormal matrices and the diagonal matrix $\Sigma
=\rm{diag}\left\{ \sigma _{1},\ldots ,\sigma _{n}\right\} $ has singular values $%
\sigma _{1}\geq \sigma _{2}\geq \ldots \geq \sigma _{n}$. If $\sigma
_{R+1},\sigma _{R+2},\ldots ,\sigma _{n}$ are as small as the magnitude of
noise level, the rank-R matrix $U_{R}\Sigma _{R}V_{R}^{\top}$ will be taken as
the matricization $T$ where $U_{R}=[
\mathbf{u}_{1},  \cdots ,  \mathbf{u}_{R}] $, $V_{R}=[
\mathbf{v}_{1},  \cdots ,  \mathbf{v}_{R}] $ and $\Sigma _{R}=\rm{diag}\left\{ \sigma _{1},\ldots ,\sigma
_{R}\right\} $. The SVD provides the full rank factorization of $T=EF^{\top}$
in the step 2 by setting $E=[
\mathbf{u}_{1},  \cdots ,  \mathbf{u}_{R}] $ and $F=V_{R}\Sigma _{R}$.

All numerical experiments are carried out in Matlab with machine precision $%
\epsilon _{machine}\approx 2.2\times 10^{-16}$. For solving the polynomial
equation in step 4, we use Matlab version HOM4PS2 \cite{Lee:2008}, which
implements the homotopy method for solving polynomial systems. The accuracy
is measured in terms of relative error $err=\Vert \mathcal{A}-\mathcal{\hat{A}}\Vert_F /\Vert \mathcal{A}\Vert_F$, where $\mathcal{\hat{A}}$ is the computed
tensor.

In the first experiment, we consider test tensors that are generated by the
rank-$4$ tensor $\mathcal{A}\in \mathbb{R}^{3\times 3\times 6}$ as \eqref{ex1} in \textbf{Example 1} and contaminated with noise as in \eqref{eq6.1}.
For a given noise level, the \textbf{Algorithm 1} runs 100 times with different
noise. Figure 1 shows the effect of varying the noise level $\theta $ on the
error $err$. In this experiment, the relative errors are roughly at
the noise level.

\begin{figure}[tbp]
\begin{center}
\epsfig{figure=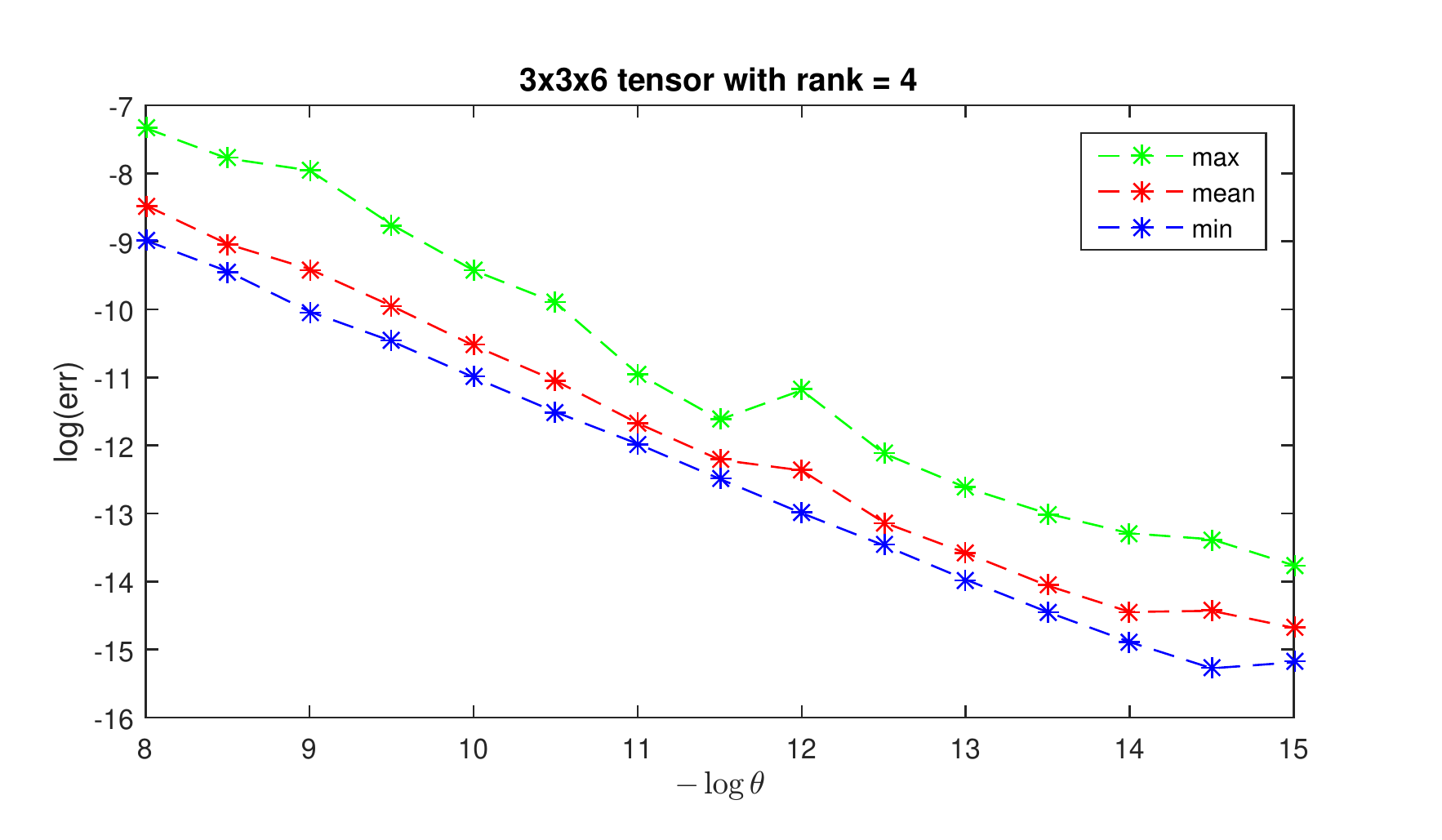,height=2.2in}
\end{center}
\caption{Relative error in the first experiment.}
\label{fig1}
\end{figure}

In the second experiment, the test tensors are generated by the rank-$28$
tensor $\mathcal{A}\in \mathbb{R}^{3\times 3\times 4\times 30}$ and are
perturbed by noise. We generate the rank-$28$ tensor by
\begin{equation*}
\mathcal{A}=\sum_{r=1}^{28}\mathbf{x}_{r}\circ \mathbf{y}%
_{r}\circ \mathbf{z}_{r}\circ \mathbf{v}_{r},
\end{equation*}%
where the entries of $\mathbf{x}_{r}$, $\mathbf{y}
_{r}$, $\mathbf{z}_{r}$ and  $\mathbf{v}_{r}$ are randomly generated with
distribution $N(0,1)$. For a given noise level, \textbf{Algorithm 2} runs 100
times with different noise. The effect of varying the noise level $\theta $
on the error $err$ is shown in Figure 2. We can see that the relative errors are roughly at
the noise level when the noise level is larger than $10^{-13}$. In both experiments, \textbf{Algorithm 1} and \textbf{Algorithm 2} can always find the CP decomposition.
\begin{figure}[h]
\begin{center}
\epsfig{figure=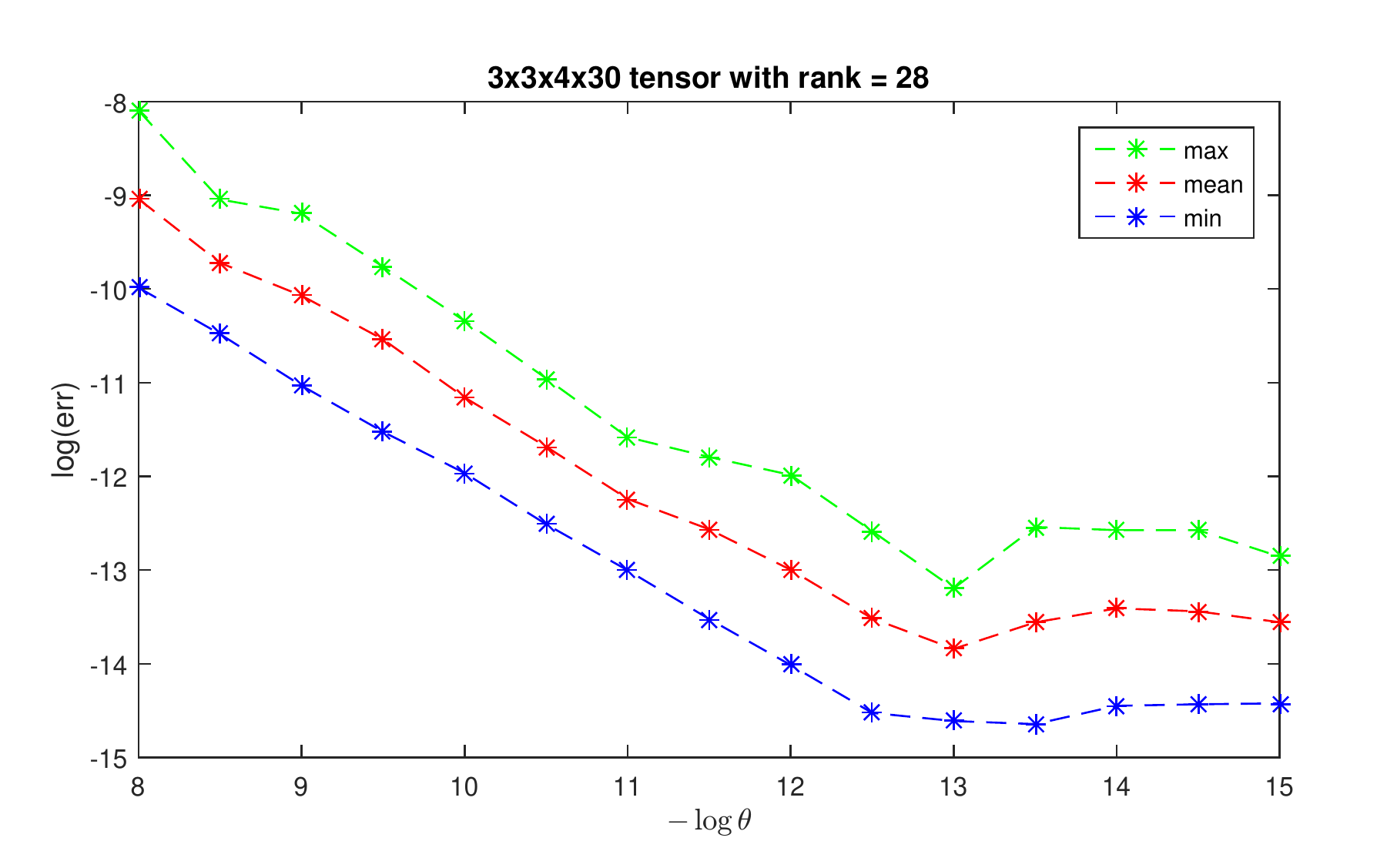,height=2.2in}
\end{center}
\caption{Relative error in the second experiment.}
\label{fig2}
\end{figure}

\bigskip

% ======================================================================

\end{document}